\documentclass{amsart}
\usepackage[utf8]{inputenc}
\usepackage{mathrsfs}
\usepackage{tikz-cd}
\usepackage[all]{xy}
\usepackage{enumitem}
\usepackage[all]{xypic}
\usepackage{amssymb, amsmath, amscd, amsthm,mathtools}
\usepackage{bbold}
\usepackage{epstopdf}
\usepackage{mathdots}
\usepackage[bookmarks=true,bookmarksnumbered=true]{hyperref}
\usepackage[hmargin=2cm,vmargin=3cm]{geometry}
\usepackage{ulem}
\usepackage[all]{xy}
\allowdisplaybreaks

\numberwithin{equation}{section}

\theoremstyle{plain}
\newtheorem{thm}{Theorem}[section]
\newtheorem{lem}[thm]{Lemma}
\newtheorem{prop}[thm]{Proposition}
\newtheorem{propt}[thm]{Property}
\newtheorem{cor}[thm]{Corollary}

\newcommand{\Mp}{\widetilde{\Sp}}
\newtheorem*{hyp}{Working Hypothesis on char $p$ case}

\theoremstyle{definition}
\newtheorem{rem}[thm]{Remark}
\newtheorem{defn}[thm]{Definition}

\newcommand{\cl}[1]{\widetilde{#1}}

\newcommand{\pair}[1]{\langle #1 \rangle}
\renewcommand{\1}{{\bf 1}}

\newcommand{\tr}{\mathrm{tr}}

\def\cha{\operatorname{char}}

\def\disc{\operatorname{disc}}

\def\Hom{\operatorname{Hom}}
\def\ind{\operatorname{ind}}
\def\Ind{\operatorname{Ind}}
\def\Irr{\operatorname{Irr}}

\def\Re{\operatorname{Re}}
\def\Res{\operatorname{Res}}

\def\temp{\operatorname{temp}}

\def\beq{\begin{equation}}
\def\eeq{\end{equation}}
\def\beqn{\begin{equation*}}
\def\eeqn{\end{equation*}}
\def\beqna{\begin{eqnarray}}
\def\eeqna{\end{eqnarray}}
\def\beqnan{\begin{eqnarray*}}
\def\eeqnan{\end{eqnarray*}}

\def\wt{\widetilde}

\def\rm{\mathrm}
\def\mc{\mathcal}

\def\mf{\mathfrak}

\def\bs{\backslash}
\def\an{\mathrm{an}}

\def\GL{\mathrm{GL}}

\def\O{\mathrm{O}}
\def\SO{\mathrm{SO}}

\def\Sp{\mathrm{Sp}}
\def\G{\mathrm{G}}
\def\H{\mathrm{H}}
\def\J{\mathrm{J}}
\def\U{\mathrm{U}}
\def\P{\mathrm{P}}
\def\B{\mathrm{B}}
\def\K{\mathrm{K}}
\def\R{\mathrm{R}}
\def\N{\mathrm{N}}
\def\M{\mathrm{M}}
\def\T{\mathrm{T}}
\def\Q{\mathrm{Q}}
\def\A{\mathbb{A}}
\def\Z{\mathrm{Z}}

\def\e{\epsilon}
\def\AA{\mathbb{A}}
\def\CC{\mathbb{C}}

\def\FF{\mathbb{F}}
\def\VV{\mathbb{V}}
\def\WW{\mathbb{W}}
\def\NN{\mathbb{N}}

\def\HH{\mathbb{H}}

\def\RR{\mathbb{R}}
\def\QQ{\mathbb{Q}}

\def\ss{\subset}

\def\la{\langle}
\def\ra{\rangle}

\def\bs{\backslash}

\title{The local converse theorem for quasi-split $\O_{2n}$ and $\SO_{2n}$}
\author{Jaeho Haan, Yeansu Kim, Sanghoon Kwon}


\address{Department of Mathematics Education\\
Catholic Kwandong University\\
Gangneung, Korea}
\email{jaehohaan@gmail.com}

\address{
77 Yongbong-ro, Buk-gu \\
Department of Mathematics Education\\
Chonnam National University\\
Gwangju, Korea}
\email{ykim@jnu.ac.kr}

\address{Department of Mathematics Education\\
Catholic Kwandong University\\
Gangneung, Korea}
\email{skwon@cku.ac.kr}

\begin{document}
\maketitle
\begin{abstract}
Let $F$ be a non-archimedean local field of characteristic not equal to 2. In this paper, we prove the local converse theorem for quasi-split $\O_{2n}(F)$ and $\SO_{2n}(F)$, via the description of the local theta correspondence between $\O_{2n}(F)$ and $\Sp_{2n}(F)$. More precisely, as a main step, we explicitly describe the precise behavior of the $\gamma$-factors under the correspondence. Furthermore, we apply our results to prove the weak rigidity theorems for irreducible generic cuspidal automorphic representations of $\O_{2n}(\A)$ and $\SO_{2n}(\mathbb{A})$, respectively, where $\A$ is a ring of adele of a global number field $L$.
\end{abstract}

\section{Introduction}
Let $F$ be a non-archimedean local field of odd residual characteristic, that is, a finite extension of either $\QQ_p$ for an odd prime $p$ or $\FF_{q}(\!(t)\!)$ for an odd prime power $q$. Let $V$ be a $2n$-dimensional symmetric space over $F$, let $\O(V)$ be its orthogonal group (the isometry group of $V$) and $\SO(V)$ be the special orthogonal group of $V$, the identity component of $\O(V)$. We fix an element $\e$ in $\O(V) / \SO(V)$. (See Subsection~\ref{e} for the precise definition.) We write $\G_n:=\O(V)$ and $\H_n:=\SO(V)$, and assume that both groups are quasi-split. Let $\U$ be the maximal unipotent subgroup of a fixed Borel subgroup of $\SO(V)$ and let $\wt{\U}=\U \rtimes \la \e \ra.$ Let $\wt{\mu}$ (resp. $\mu$) be a generic character of $\wt{\U}(F)$ (resp. $\U(F)$).
In this paper, we prove the following local converse theorem for quasi-split even orthogonal groups under the following hypothesis on $\cha(F)=p\ne2$ case.

\begin{hyp}
The $\gamma$-factors for $\Sp_{2n} \times \GL_l$ are properly defined in $\cha(F)=p\ne2$ cases. Furthermore, they satisfy natural properties of $\gamma$-factors. (For the precision of the natural properties, see Property~\ref{Property_gamma} (i)-(vii) 
except (vi).)
\end{hyp}

\begin{thm}[Local Converse Theorem for $\O_{2n}$]\label{thm:main} Assume the above working hypothesis on $\cha(F)=p\ne2$ case.
Let $\wt{\pi_1}$ and $\wt{\pi_2}$ be irreducible $\wt{\mu}$-generic representations of $\G_n(F)$ with the same central characters such that
\[\gamma(s,\wt{\pi_1} \times \rho,\psi)=\gamma(s,\wt{\pi_2} \times \rho,\psi)
\] holds for any irreducible supercuspidal representation $\rho$ of $\GL_i(F)$ with $1\le i \le n$. Then we have
\[\pi_1 \simeq \pi_2. 
\]
Here, the $\gamma$-factors are defined as the local gamma factors of $\G_n \times \GL_i$ (See Definition~\ref{Ogamma}.)
\end{thm}

From this, we deduce immediately the following theorem.

\begin{thm}[Local Converse Theorem for $\SO_{2n}$] Assume the above working hypothesis on $\cha(F)=p\ne2$ case. Let $\pi_1$ and $\pi_2$ be irreducible $\mu$-generic representations of $\H_n(F)$ with the same central characters such that
\[\gamma(s,\pi_1 \times \rho,\psi)=\gamma(s,\pi_2 \times \rho,\psi)
\] holds for any irreducible supercuspidal representation $\rho$ of $\GL_i(F)$ with $1\le i \le n$. Then we have
\[\pi_1 \simeq \pi_2 \quad \text{or} \quad  \pi_1 \simeq \pi_2^{\e}. 
\]
Here, the $\gamma$-factors are the local gamma factors of $\H_n \times \GL_i$ defined by Langlands-Shahidi method or Rankin--Selberg method and $\pi_2^{\e}$ is the conjugation of $\pi_2$ by $\e$.
\end{thm}

There are various versions of the local converse theorem that apply to different groups. To our knowledge, it seems that Henniart was the first to address this type of problem. In his work \cite{He1}, he proved a weak version of the Local Converse Theorem (LCT) for $\GL_n$. More precisely, he showed that $\pi_1$ and $\pi_2$ are equivalent if $\gamma(s,\pi_1\times\sigma,\psi)=\gamma(s,\pi_2\times\sigma,\psi)$ for any generic irreducible admissible representation $\sigma$ of $\GL_{n-1}(F)$ when $\textrm{char}(F)=0$. This was extended by Jiang and Soudry \cite{JS03} to the case of odd special orthogonal groups $\SO_{2n+1}$ using the global weak functorial lift from $\SO_{2n+1}$ to $\GL_{2n}$, the local descent from $\GL_{2n}$ to $\Mp_{2n}$, and the Howe lifts from $\SO_{2n+1}$ to $\Mp_{2n}$. They also applied a local-to-global argument to extend Henniart's result to the case of $\SO_{2n+1}$. After Jiang \cite{Jng} proposed the LCT for all classical groups, Chai \cite{Cha19}, Jacquet and Liu \cite{JL18} and  P. Yan, Q. Zhang \cite{YZ23} independently improved upon Henniart's result to the best.

Quite recently, Zhang \cite{Q18, Q19} proved the supercuspidal cases of the conjectures for $\Sp_{2n}$ and $\U_{2n+1}$ using a theory of partial Bessel functions developed by Cogdell, Shahidi and Tsai in \cite{CST17}. Following in the same vein as Zhang's work, Jo \cite{Jo} extended the results of Jiang-Soudry \cite{JS03} and Zhang \cite{Q18} to generic cases for $\SO_{2n+1}$ and $\Sp_{2n}$ for non-archimedean local fields of characteristic different from 2. For other classical groups, Morimoto \cite{M18} proved the theorem for even unitary groups using the local descent method. Recently, Hazeltine and Liu \cite{HL} announced that they have also proven it for split $\SO_{2n}$ using a similar method. Hazeltine extended it for quasi-split $\SO_{2n}$ over a local field of characteristic zero and a finite field using the local descent method in \cite{Ha} (see \cite[Introduction]{HL} for more details). It is worth noting that most of the literature cited above deals with the local field of characteristic zero case, except for \cite{Jo}.

The main purpose of the paper is to prove the LCT for quasi-split $\SO_{2n}$ with arbitrary type $(d,c)$ (see subsection \ref{orthogonal} for the definition of type $(d,c)$) and quasi-split $\O_{2n}$ when a local field $F$ is of any characteristic not equal to $2$. This can be done by thoroughly studying the local theta correspondence between $\O_{2n}$ and $\Sp_{2n}$. Specifically, we establish a precise relationship between the $\gamma$-factors of $\O_{2n}$ and those of $\Sp_{2n}$, which enables us to settle the LCT for quasi-split $\O_{2n}$. Using the restriction method from $\O_{2n}$ to $\SO_{2n}$, we then deduce the LCT for the quasi-split $\SO_{2n}$ case. Let us remark that in the case when $F$ is of characteristic zero, this is a slight generalization of Hazeltine and Liu \cite{HL}, which proves the LCT for split $\SO_{2n}$ of type $(1, -1/8)$.

\begin{rem}
\begin{enumerate}
    \item When $F$ is a local field of characteristic 0, Theorem~\ref{thm:main} follows relatively easily once we use Arthur's results \cite{Ar} on the local Langlands correspondence for quasi-split $\O_{2n}$, as explained in \cite{AG17}, together with the LCT for $\GL_{2n}$ \cite{He1} and the uniqueness of generic members in $L$-packets \cite{At17}. Note that our method is independent of Arthur's results. We believe that the proofs independent of Arthur's results have intrinsic value.
    \item The first author has also written a separate paper \cite{H} on the LCT for the metaplectic groups. We present this work separately for two reasons: to avoid making the notation too complicated, and because the metaplectic case can be achieved using existing results in the literature, while the even orthogonal case requires more non-trivial substantial work. We also note that \cite{YZ24} established a refined version of a local converse theorem for the group $\SO_4$ over a $p$-adic field $F$. Specifically, they showed that a generic supercuspidal representation $\pi$ of $\SO_4(F)$ is uniquely determined by its $\GL_1$, $\GL_2$-twisted local gamma factors, together with a twisted exterior square local gamma factor of $\pi$.    
\end{enumerate}
\end{rem}

As an application of our main results, the second purpose of the paper is to prove the rigidity theorem for $\O_{2n}$ over a global number field.

\begin{thm} [Weak rigidity theorem for $\O_{2n}$]\label{weak}
Let $L$ be a global number field and $\AA$ its adele ring. Let $\wt{\chi}$ be a non-trivial generic  character of $\wt{\U}(F)\bs \wt{\U}(\A)$ and $\wt{\pi}=\otimes_v \wt{\pi_v}$ and $\wt{\pi}'=\otimes_v \wt{\pi_v}'$ be irreducible cuspidal $\wt{\chi}$-generic automorphic representations of $\O_{2n}(\A)$. If $\wt{\pi_v}\simeq \wt{\pi_v}'$ or $\wt{\pi_v}\simeq \wt{\pi_v}'\otimes \det$  for almost all $v$, then $\wt{\pi_v}\simeq \wt{\pi_v}'$ or $\wt{\pi_v}\simeq \wt{\pi_v}'\otimes \det$ for all places of $L$.
\end{thm}

The weak rigidity theorem can be obtained using Arthur's multiplicity formula for $\O_{2n}$, as explained in \cite{AG17}. Again, we stress that our proof of Theorem~\ref{weak} is independent of Arthur's results. 

This paper is organized as follows. In Section \ref{Prelim}, we prepare the basic setup to state our theorem precisely. We explain the local theta correspondence and review some of its relevant results in Section~\ref{theta}.
In Section~\ref{gammacf}, we examine the relationship of $\gamma$-factors for $\O_{2n}$ and $\Sp_{2m}$ through the local theta correspondence for $(\O_{2n},\Sp_{2m})$. We provide a proof of our main theorem in Section~\ref{The proof}. For some technical issues that arise in the local theta correspondence theory, we divide the proof into two steps: we first prove the tempered case, and then we generalize it to the generic case. In Section~\ref{sec:applications}, we demonstrate the weak rigidity theorem for quasi-split $\O_{2n}$ and $\SO_{2n}$ as an application of the theorems established in the previous sections. The proofs of foundational results that are essential for establishing Proposition~\ref{Theta1} and Theorem~\ref{Theta2} are contained in Appendix~\ref{sec:tjm} and Appendix~\ref{sec:gtl}.

\subsection{Notations}
\begin{itemize}
\item $F$ : a non-archimedean local field of characteristic different from 2
\item $\cha(F)$ : the characteristic of $F$
\item $q$ : the order of the residue field of $F$ such that $q=p^e$ for some odd prime $p$
\item $\psi$ : a non-trivial additive character of $F$
\item $\varpi$ : a uniformizer of $F$
\item $|\cdot|_F$ : a normalized absolute value in $F$ such that $|\varpi|_F=q^{-1}$
\item $\Irr(G)$ : the set of isomorphism classes of irreducible smooth representations of $G:=\G(F)$, where $\G(F)$ is the $F$-points of a reductive group $\G$ defined over $F$
\item $\Irr_{\temp}(G)$ : the set of isomorphism classes of irreducible smooth tempered representations of $G:=\G(F)$, where $\G(F)$ is the $F$-points of a reductive group $\G$ defined over $F$
\item $\pi^{\vee}$ : the contragredient representation of $\pi \in \Irr(\G(F))$
\item $V=V_n$ : a $2n$-dimensional orthogonal space, i.e., a $2n$-dimensional vector space over $F$ equipped with a non-degenerate symmetric bilinear form $(\ ,\ )_{V_n}$
\item $\G_n:=\O(V_n)$ : the orthogonal group of $V_n$ 
\item $\H_n:=\SO(V_n)$ : the special orthogonal group of $V_n$ 
\item $W_m$ : a $2m$-dimensional symplectic space over $F$, i.e., a $2m$-dimensional vector space over $F$ equipped with a non-degenerate symplectic bilinear form $\la \ , \ \ra_{W_m}$

\item $\J_m:=\Sp(W_m)$ : the symplectic group of $W_m$
\item $\SO_{n} := \SO(V_n)$ for some $V_n$
\item $\Sp_{2m} := \Sp(W_m)$ for some $W_m$

\item $\mathbb{I}$ : the trivial representation
\item $N_{E/F}$ : a norm map from $E$ to $F$, where $E$ is a quadratic extension of $F$
\item $\Ind_{B}^G$ : the unnormalized induction for an algebraic group $G:=\G(F)$ and its closed subgroup $B$
\item $\ind_{B}^G$ : the unnormalized compactly supported induction for an algebraic group $G:=\G(F)$ and its closed subgroup $B$
\item $\mc{A}(G)$ : the space of automorphic forms on $G$ for an algebraic group $G$
\item $\mc{A}_{cusp}(G)$ : the space of cusp forms on $G$ for an algebraic group $G$
\item $\mu_2$ : an algebraic group of order 2
\item $\textbf{1}$ : the identity element in $\G(F)$ for a given algebraic group $\G$
\end{itemize}

\section{Preliminaries}\label{Prelim}
In this section, we prepare the basic setup to state our main theorem. While many of the theorems and propositions we quote in this paper only address the case where $\cha(F)=0$, most of the proofs apply equally to the positive characteristic cases. Therefore, instead of repeating the same proof in the references for the case $\cha(F)=0$, we adopt the position that they also cover the case $\cha(F)=p$ when it is applicable. However, when the proof of a proposition or a theorem for $\cha(F)=0$ is not clearly stated in the references, we provide a proof that covers both cases.

\subsection{Orthogonal and symplectic groups}
\subsubsection{Orthogonal groups}\label{orthogonal}
Let $V_n$ (or $V$ if there is no confusion on the dimension) be a $2n$-dimensional vector space over $F$ equipped with a non-degenerate symmetric bilinear form $(\cdot,\cdot)_V$ and let $\G_n:=\O(V)$ and $\H_n:=\SO(V)$ be the associated orthogonal and special orthogonal group, respectively. When $n=0$, we define $\G_0$ and $\H_0$ as the trivial group.

Define the discriminant of $V$ by
\[
\disc(V)=\disc(V,(\cdot,\cdot)_V)
=2^{-n}
(-1)^{\frac{n(n-1)}{2}}\det((e_i,e_j)_V)_{i,j})\bmod F^{\times2}\in F^\times/F^{\times2}.
\]
where $\{e_1, \cdots, e_{2n} \}$ is a basis of $V$.
Let $\chi_V=(\cdot,\disc(V))_{\mf{h}}$ be the quadratic character of $F^{\times}$ associated with $F(\sqrt{\disc(V)})/F$, where $(\cdot,\cdot)_{\mf{h}}$ is the quadratic Hilbert symbol. 

Denote by $V_{\an}$ the anisotropic kernel of $V$. It is easy to check that $\H_n$ 
and $\G_n$ are quasi-split if and only if $\dim(V_{\an}) \le 2$. 

In this case, let $\{e_1,\cdots,e_{n-1}\}$  and $\{e_1^*,\cdots,e_{n-1}^*\}$ be subsets of $V$ satisfying
\[(e_i,e_j)_V=(e_i^*,e_j^*)_V=0, \quad (e_i,e_j^*)_V=\delta_{ij}.
\]
For $1\le k \le n-1$, let 
\[X_k=\text{Span}\{e_1,\cdots,e_k\} \text{ and } X_k^*=\text{Span}\{e_1^*,\cdots,e_k^*\},
\] and $V_{n,k}$ be the orthogonal complement of $X_k \oplus X_k^*$ in $V$ so that $V=X_k \oplus V_{n,k}\oplus X_k^*$. 

Next, we consider the flag of isotropic subspaces
\[X_{k_1} \ss X_{k_1+k_2} \ss \cdots \ss X_{k_1+\cdots +k_r}\ss V.
\]
The stabilizer of such a flag is a parabolic subgroup $\P$ of $\H_n$ whose Levi factor $\M$ is
\[{ \M} \simeq \GL_{k_1} \times \cdots \times \GL_{k_r} \times \H_{n-k_1-\cdots - k_r},
\]
where each $\GL_{k_i}$ is the group of invertible linear maps on Span $\{e_{k_i+1},\cdots,e_{k_{i+1}}\}$.
It is known that there exist $c,d \in F^{\times}$ such that 
\begin{equation}\label{type}
V_{n,n-1} \cong F[X]/(X^2-d)
\end{equation}
becomes a $2$-dimensional vector space equipped with the pairing \[
(\alpha,\beta)\mapsto \pair{\alpha,\beta}_{V_{n,n-1}}\coloneqq c\cdot\tr(\alpha \cdot \e(\beta)),
\]
where $\e$ is the involution on $F[X]/(X^2-d)$
induced by $a+bX\mapsto a-bX$. The images of $1,X \in F[X]$ through the isomorphism (\ref{type}) are denoted by $e,e'$, respectively.
In this setting, i.e., $V_n=X_{n-1} \oplus V_{n,n-1} \oplus X_{n-1}^*$ with the isomorphism (\ref{type}), 
we say that $V_n$ has a type $(d,c)$. Note that when $V_n$ have the types $(d,c)$ and $(d',c')$ for $c, c', d, d' \in F^{\times}$, then 
\[ d =d' \pmod{F^{\times2}}, \quad c =c' \pmod{N_{E/F}(E^{\times})},\] 
where $E=F(\sqrt{d})$.\par
Throughout the paper, when $V_n$ appears, we always assume that $\G_n$ and $\H_n$ are quasi-split, $V_n$ has a type $(d,c)$ and 
$c'$ is any element in $ cN_{E/F}(E^{\times})/F^{\times2}$ for fixed $c, d \in F^{\times}$.

\subsubsection{Symplectic groups}
Let $H$ be the hyperbolic plane over $F$, i.e. the split symplectic space of dimension 2,
and for $r \ge 0$, let $W_r=H^{\oplus r}$ and $\la \ ,\ \ra_{W_m}$ the non-degenerate symplectic form of $W_m$. The collection $\{W_r \ \vert \ r  \ge 0\}$ is called a Witt tower of spaces. For each $m \in \NN$, let $\J_m:=\Sp(W_m)$ be the associated symplectic group. When $m=0$, we define $\J_0$ as the trivial group.
Let $\{f_1,\cdots,f_m,f_1^*,\cdots,f_m^*\}$ be a specific basis of $W_m$ satisfying
\[\pair{f_i,f_j}_{W_n}=\pair{f_i^*,f_j^*}_{W_n}=0, \quad \pair{f_i,f_j^*}_{W_n}=\delta_{ij}.
\]
For $1\le k \le m$, let 
\[Y_k=\text{Span}\{f_1,\cdots,f_k\} \text{ and } Y_k^*=\text{Span}\{f_1^*,\cdots,f_k^*\},
\]
so that $W_m=Y_m \oplus Y_m^*$. We also set
\[W_{m,k}=\text{Span}\{f_{k+1},\cdots,f_m,f_m^*,\cdots,f_{k+1}^*\},
\]
so that $W_m=Y_k \oplus W_{m,k} \oplus Y_k^*$.\\

Next, we consider the flag of isotropic subspaces
\[Y_{k_1} \ss Y_{k_1+k_2} \ss \cdots \ss Y_{k_1+\cdots +k_r}\ss W_m.
\]
The stabilizer of such a flag is a parabolic subgroup $\P$ of $\J_m$ whose Levi factor $\M$ is
\[\M \simeq \GL_{k_1} \times \cdots \times \GL_{k_r} \times \J_{m-k_1-\cdots - k_r},
\]
where each $\GL_{k_i}$ is the group of invertible linear maps on Span $\{f_{k_i+1},\cdots,f_{k_{i+1}}\}$.

\subsection{Representations of $\SO(V)$ and $\O(V)$}\label{e}
Let $V$ be a $2n$-dimensional orthogonal space over $F$. For $n \ge 1$, suppose that $\G_n:=\O(V)$ and $\H_n:=\SO(V)$ are quasi-split. Decompose $V$ as $V=X_{n-1} \oplus V_{n,n-1} \oplus X_{n-1}^*$  and fix $\e \in \O(V_{n,n-1}) \bs \SO(V_{n,n-1} )$ such that $\e(e)=e, \e(e')=-e'$. Through the natural embedding $\O(V_{n,n-1} ) \hookrightarrow \O(V)$, we may regard $\e$ as an element in $\O(V)$ which acts trivially on $X_{n-1} \oplus X_{n-1}^*$. For $\pi \in \Irr(\H_n(F))$, denote by $\pi^{\e}$ the conjugation of $\pi$ by $\e$.

The next proposition follows from the Clifford theory
(e.g., see \cite[Lemma 4.1]{BJ}).
\begin{prop} [{\cite[Proposition 2.1]{At18}}]\label{Clifford}

The following holds true.

\begin{enumerate}
\item\label{epinv}
For $\pi\in\Irr(\H_n(F))$, the following are equivalent:
\begin{itemize}
\item
$\pi^\e\cong\pi$;
\item
there exists $\cl{\pi}\in\Irr(\G_n(F))$ such that $\cl{\pi}|\H_n(F) \cong \pi$;
\item
the induction $\Ind_{\H_n(F)}^{\G_n(F)}(\pi)$ is reducible;
\item
$\Ind_{\H_n(F)}^{\G_n(F)}(\pi)$ $\cong \cl\pi\oplus(\cl\pi\otimes\det)$
for any $\cl\pi\in\Irr(\G_n(F))$ with $\cl{\pi}|\H_n(F)\cong \pi$.
\end{itemize}
\item
For $\cl{\pi}\in\Irr(\G_n(F))$, the following are equivalent:
\begin{itemize}
\item
$\cl{\pi}\otimes\det\cong\cl{\pi}$;
\item
there exists $\pi\in\Irr(\H_n(F))$ such that $\Ind_{\H_n(F)}^{\G_n(F)}(\pi)$ $\cong \cl\pi$;
\item
the restriction $\cl{\pi}|\H_n(F)$ is reducible;
\item
$\cl\pi|\H_n(F)\cong \pi\oplus\pi^\e$
for any $\pi\in\Irr(\H_n(F))$ with $\Ind_{\H_n(F)}^{\G_n(F)}(\pi)\cong\cl\pi$.
\end{itemize}
\end{enumerate}
\end{prop}
For $\pi\in\Irr(\H_n(F))$, we denote the equivalence class of $\pi$ by
\[
[\pi]=\{\pi,\pi^\e\}\in\Irr(\H_n(F))/{\sim_\e}.
\]
\subsection{Generic characters and generic representations}\label{subsec:gcgr}
\subsubsection{Generic characters}
Throughout the paper, let $\psi$ be a fixed non-trivial additive character of $F$. Let $\U$ (resp. $\U'$) be the unipotent radical of a Borel subgroup $\B=\T\U$ (resp. $\B'=\T'\U'$) of $\H_n$ (resp. $\J_m$), where $\T$ (resp. $\T'$) is the $F$-rational torus stabilizing the lines $Fe_i$ (resp. $Ff_i $) for each $i=1,\cdots,n-1$ (resp. $i=1,\cdots,m$).

A character $\wt{\chi}$ of $\U(F)$ and a character $\chi$ of $\U'(F)$ are called to be generic if the stabilizer of $\wt{\chi}$ in $\T(F)$ and the stabilizer of $\chi$ in $\T'(F)$ are both equal to the center of $\H_n(F)$ and $\J_m(F)$, respectively. By utilizing the structure of $V$, we can define several special (and essentially all) generic characters of $\U(F)$ and $\U'(F)$.

Suppose that $V$ has type $(d,c)$ for some $c,d \in F^{\times}$. Write $E=F(\sqrt{d})$. For an arbitrary $c' \in cN_{E/F}(E^{\times})/ F^{\times2}$, we define a generic character $\mu_{c'}$ of $\U(F)$ by
\[\mu_{c'}(u)=\psi((ue_2,e_1^*)_V+\cdots + (ue_{n-1},e_{n-2}^*)_V+(ue,e_{n-1}^*)_V), \quad u\in \U(F)
\]in which we regard $V$ has a type $(d,c')$.

\noindent By \cite[Sect. 12]{GGP}, the map $c'\mapsto \mu_{c'}$ gives a bijection (not depending on $\psi$)
\[
cN_{E/F}(E^\times)/F^{\times2}\rightarrow \text{\{$\T(F)$-orbits of generic characters of $\U(F)$\}}.
\]
Note that $\e$ normalizes $\U$ and fixes $\mu_{c'}$. Put $\wt{\U}=\U \rtimes \la \e \ra$. Since $\e$ fixes $\mu_{c'}$, we can extend $\mu_{c'}$ to $\wt{\U}(F)$. There are exactly two extensions $\mu_{c'}^{+},\mu_{c'}^{-}:\wt{\U}(F) \to \CC^{\times}$ which are determined by 
\[\mu_{c'}^{\pm}(\e)=\pm{1}.
\]

On the other hand, for $d' \in F^{\times} / F^{\times2}$, define a generic character $\mu_{d'}'$ of $\U'(F)$ by
\[\mu_{d'}'(u')=\psi(\pair{u'f_2,f_1^*}_{W_m}+\cdots + \pair{u'f_m,f_{m-1}^*}_{W_m}+d\pair{u'f_m^*,f_m^*}_{W_m}), \quad u' \in \U'(F).
\]
Again by \cite[Sect. 12]{GGP}, the map $d'\mapsto \mu_{d'}'$ gives a bijection (depending on $\psi$)
\begin{equation}\label{eq:t'-orbits}
F^\times/F^{\times2}\rightarrow \text{\{$\T'(F)$-orbits of generic characters of $\U'(F)$\}}.
\end{equation}

\subsubsection{Generic representations}
\noindent For $\pi \in \Irr(\H_n(F))$ (resp. $\wt{\pi} \in \Irr(\G_n(F))$), if $\Hom_{\U(F)}(\pi,\mu_{c'}) \ne 0$ (resp. $\Hom_{\wt{\U}(F)}(\wt{\pi},\mu_{c'}^{\pm}) \ne 0$), then we say that $\pi$ (resp. $\wt{\pi}$) is $\mu_{c'}$-generic (resp. $\mu_{c'}^{\pm}$-generic.) Note that if $\wt{\pi}$ is $\mu_{c'}^{\pm}$-generic, then $\wt{\pi} \otimes \det$ is $\mu_{c'}^{\mp}$-generic. Similarly, for $\tau \in \Irr(\J_m(F))$, if $\Hom_{\U'(F)}(\tau,\mu_{d'}') \ne 0$, then we say that $\tau$ is $\mu_{d'}'$-generic. 

Note that in this subsection we follow the setting and notations in \cite[Section~{7.3}]{AG17} and parts of ideas in the proof of the results in this subsection are inspired from \cite{AG17}.

The following lemma illustrates the transformation of genericity through the induction functor $\Ind_{\H_n(F)}^{\G_n(F)}$.

\begin{lem}[{\cite[Lemma 2.3]{AG17}}]\label{gen}
Let $\pi \in \Irr(\H_n(F))$. 
\begin{enumerate}
\item
Assume that $\pi$ can be extended to $\G_n(F)$.
Then there are exactly two such extensions.
Moreover, the following are equivalent:
\begin{enumerate}
\item[(A)]$\pi$ is $\mu_{c'}$-generic;
\item[(B)]
exactly one of two extensions is $\mu_{c'}^+$-generic but not $\mu_{c'}^-$-generic,
and the other is $\mu_{c'}^-$-generic but not $\mu_{c'}^+$-generic.
\end{enumerate}

\item
Assume that $\pi$ can not be extended to $\G_n(F)$.
Then $\wt{\pi} = \Ind_{\H_n(F)}^{\G_n(F)}(\pi)$ is irreducible.
Moreover, the following are equivalent:
\begin{enumerate}
\item[(A)]$\pi$ is $\mu_{c'}$-generic;
\item[(B)]
$\wt{\pi}$ is both $\mu_{c'}^+$-generic and $\mu_{c'}^-$-generic.
\end{enumerate}

\end{enumerate}
\end{lem}

There exists a global analogue of the notion of generic representations. Consider a global field $L$ and its adèle ring $\A$. One can define the algebraic groups $\G_n, \H_n, \U$ over $L$ in a similar manner as in the local case. For an automorphic (i.e., left $\U(L)$-invariant) generic character $\chi$ of $\U(\A)$ and an automorphic form $\varphi$ of $\H_n(\A)$, denote the $\chi$-Fourier coefficient of $\varphi$ as
\begin{equation*} 
W_{\varphi}^{\chi}(g)\coloneqq \int_{\U(L) \bs \U(\A)} \varphi(u g)\cdot \overline{\chi(u)} du.
\end{equation*}
For an automorphic representation $\pi$ of $\H_n(\A)$, $\pi$ is called globally $\chi$-generic if $W_{\varphi}^{\chi}(\1) \neq 0$ for some $\varphi \in \pi$.

For each place $v$, let $\K_v$ be a maximal compact subgroup of $\G_n(L_v)$ such that $\K_v$ is special if $v$ is non-archimedean. For each place $v$ of $L$, take $\e_v \in \K_v$ such that $\e_v^2=\textbf{1}$,
$\det(\e_v )=-1$, $\e_v$ stabilizes $\chi_v$ and that $\wt{\e}\coloneqq(\e_v)_v\in  \G_n(\A)$ is in $\G_n(F)$. 

For $\textbf{t}=(t_v)_v \in  \mu_2(\A)$, define $\textbf{t} \cdot \wt{\e}\in \G_n(\A)$ as
\[\begin{cases}
(\textbf{t} \cdot \wt{\e})_v=\e_v, \quad \text{if } t_v=-1,\\
(\textbf{t} \cdot \wt{\e})_v=\textbf{1}, \quad \text{ if } t_v=1.
\end{cases}\]
Since $\mu_2(\A)\cdot \wt{\e}$ stabilizes $\U(\A)$, put $\wt{\U}\coloneqq  \U \rtimes \mu_2\cdot \wt{\e}$. For an automorphic (i.e., left $\wt{\U}(L)$-invariant) character $\wt{\chi}$ of $\wt{\U}(\A)$, we say that $\wt{\chi}$ is generic if $\wt{\chi} \vert_{\U(\A)}$ is generic. 

Fix a maximal compact subgroup $\K=\prod_v \K_v$ of $\G_n(\A)$ and let $\K_1 \coloneqq \K \cap \U(\A)$, $\K_2=\K \cap (\mu_2(\A)\cdot \wt{\e})$. Define Haar measures $du$ (resp. $dt$) on $\U(\A)$ (resp. $\mu_2(\A)$) such that $\rm{vol}(\K_1,du)=\rm{vol}(\K_2,d\textbf{t})=1$.

Now we recall the precise definition of automorphic forms on $\G_n(\A)$ \cite[Section 6.7]{AG17}:
\begin{defn} We say that a function
$$
\varphi: \G_n(\A) \rightarrow \mathbb{C}
$$
is an automorphic form on $\G_n(\A)$ if $\varphi$ satisfies the following conditions:
\begin{enumerate}
    \item $\varphi$ is smooth and of moderate growth;
    \item $\varphi$ is left $\G_n(L)$-invariant;
    \item $\varphi$ is right $\K$-finite
    \item $\varphi$ is $\mathfrak{z}$-finite, where $\mathfrak{z}$ is the center of the universal enveloping algebra of $\operatorname{Lie}\left(\G_n(\K_{\infty}) \otimes_{\mathbb{R}} \mathbb{C}\right)$
\end{enumerate}
\end{defn}

Since there is no definition of cusp forms and global genericity on $\G_n(\A)$ in \cite{AG17}, we define
those concepts precisely as follows:
\begin{defn}\label{cuspO}
For an automorphic form $\varphi$ on $\G_n(\A)$, we say $\varphi$ is a cusp form if
$$
\int_{\N(F) \backslash \N(\mathbb{A})} \varphi(n) d n=0
$$
for all nontrivial unipotent subgroups $\N$ of all standard parabolic subgroups $\P=\M\N \subset \H_n.$ If an automorphic representation $\wt{\pi}$ of $\G_n(\A)$ occurs in the space of cusp forms of $\G_n(\A)$, we call $\wt{\pi}$ a cuspidal representation of $\G_n(\A)$.  
\end{defn}

Write $\mc{A}(\G_n)$ (resp. $\mc{A}_{cusp}(\G_n)$) for the space of automorphic (resp. cusp) forms of $\G_n(\A)$. For an automorphic generic character $\wt{\chi}$ of $\wt{\U}(\A)$, we define globally $\wt{\chi}$-generic automorphic representations of $\G_n(\A)$ as follows.
\begin{defn}\label{globalgenericO}
 Let $\wt{\chi}$ be an automorphic generic character of $\wt{\U}(\A)$. Assume that $\wt{\pi}$ is an irreducible cuspidal representation of $\G_n(\A)$ in $\mc{A}(\G_n)$ and let $\wt{\varphi}$ be an cusp form in $\wt{\pi}$. Let
\begin{equation} \label{ggen}
W_{\wt{\varphi}}^{\wt{\chi}}(g)=\int_{\mu_2(L) \bs \mu_2(\A)}\int_{\U(L) \bs \U(\A)} \wt{\varphi}(u(\textbf{t} \cdot \wt{\e}) g)\cdot \overline{\wt{\chi}(u(\textbf{t} \cdot \wt{\e}))} du d\textbf{t}.
\end{equation}
Then $\wt{\pi}$ is called globally $\wt{\chi}$-generic if $W_{\wt{\varphi}}^{\wt{\chi}}(\1) \neq 0$ for some $\wt{\varphi} \in \wt{\pi}$. If $\wt{\pi}$ is globally $\wt{\chi}$-generic for some automorphic generic character $\wt{\chi}$ of $\wt{\U}(\A)$, we say that $\wt{\pi}$ is globally generic.
\end{defn}

\begin{prop}\label{regen}
Let $\wt{\pi}$ be an irreducible globally $\wt{\chi}$-generic cuspidal representation of $\G_n(\A)$. Then there exists an irreducible globally $\chi$-generic cuspidal representation $\pi$ of $\H_n(\A)$ that appears in $\wt{\pi} |_{\H_n(\A)}$.  
\end{prop}
\begin{proof}
$W_{\wt{\varphi}}^{\wt{\chi}}(g)$ absolutely converges for all $\wt{\varphi} \in \wt{\pi}$ and $g \in \G_n(\A)$ because $\U(L)\bs \U(\A)$ and $\mu_2(L) \bs \mu_2(\A)$ are compact. Since $\wt{\varphi}(g)$ is right $\K$-finite, there is a finite subset $S$ of places including all archimedean places of $L$ such that $\wt{\varphi}$ is $\prod_{v \notin S} \mu_2(L_v)\cdot \wt{\e}$-invariant. 
Choose a place $v_0 \notin S$ and put $\K_{\text{fin}}\coloneqq \prod_{v \notin S \cup \{v_0\}} \mu_2(L_v)$. Then $\prod_{v \notin S} \mu_2(L_v)= \{(1, K_{\text{fin}})\} \bigsqcup \{(-1, \K_{\text{fin}})\}$. Here, $\{(1, \K_{\text{fin}})\}$ is $\{ (1, k) \mid k \in \K_{\text{fin}} \}$. 
 
Since $\prod_{v \in S} \mu_2(L_v)$ is a finite set, we can write $\prod_{v \in S} \mu_2(L_v)=\{(\pm 1,\cdots,\pm 1)\}$. Considering both $\prod_{v \in S} \mu_2(L_v)$ and $\prod_{v \notin S} \mu_2(L_v)$ as a natural subgroup of $\mu_2(\A)$, we have 
\[
\mu_2(L) \bs \mu_2(\A) \cong \bigsqcup_{a \in \prod_{v \in S} \mu_2(L_v)} \{(a, 1, \K_{\text{fin}})\}.
\]
and thus we have
\[W_{\wt{\varphi}}^{\wt{\chi}}(g)
=\sum_{a \in \prod_{v \in S} \mu_2(L_v)}\int_{ (a, 1, \K_{\text{fin}})}\int_{\U(L) \bs \U(\A)} \wt{\varphi}(u(\textbf{t} \cdot \wt{\e})g) \cdot \overline{\wt{\chi}(u(\textbf{t} \cdot \wt{\e}))} du d\textbf{t}. \]
Recall that this is a finite sum since $|\prod_{v \in S} \mu_2(L_v)|$ is finite. Therefore, if $W_{\wt{\varphi}}^{\wt{\chi}}(\textbf{1})\ne 0$, there exists some $a \in \prod_{v \in S} \mu_2(L_v)$ such that \[\overline{\wt{\chi}(a \cdot \wt{\e})}\cdot\big(\int_{ \K_{\text{fin}}}\overline{\wt{\chi}(\textbf{t} \cdot \wt{\e})} \cdot \int_{\U(L) \bs \U(\A)} \wt{\varphi}(u(a \cdot \wt{\e})(\textbf{t} \cdot \wt{\e})) \cdot \overline{\wt{\chi}(u)}dud\textbf{t}\big) \ne 0.\]
Since $\wt{\varphi}$ is $\K_{\text{fin}} \cdot \wt{\e}$-invariant, we have 
\[
\overline{\wt{\chi}(a \cdot \wt{\e})}\cdot\big(\int_{ \K_{\text{fin}}}\overline{\wt{\chi}(\textbf{t} \cdot \wt{\e})} d\textbf{t}\big) \cdot \int_{\U(L) \bs \U(\A)} \wt{\varphi}(u(a \cdot \wt{\e})) \cdot \overline{\wt{\chi}(u)}du \ne 0.
\]
As discussed in \cite[Section~{7.3}]{AG17}, we define the restriction map
\[
\Res : \mc{A}(\G_n) \to \mc{A}(\H_n), \quad \Res(\varphi)=\varphi \vert_{\H_n(\A)}.
\]
Write $\varphi=\Res\big((a \cdot \wt{\e}) \cdot \wt{\varphi} \big)$ and $\chi=\wt{\chi} \vert_{ \U(\A)}$. Then we have $W_{\varphi}^{\chi}(\textbf{1})\ne 0$. Moreover, Definition~\ref{cuspO} implies that the restriction of a cusp form on $\G_n(\A)$ to $\H_n(\A)$ is also a cusp form on $\H_n(\A)$. Therefore, there exists an irreducible cuspidal automorphic representation $\pi$ that appears in $\wt{\pi} |_{\H_n(\A)}$ such that $\varphi$ is a cusp form in the space of $\pi$. We conclude that $\pi$ is globally $\chi$-generic by definition. 
\end{proof}

Meanwhile, the converse of the above proposition holds in the following manner.
\begin{prop}\label{exgen}Let $\pi$ be an irreducible globally $\chi$-generic cuspidal representation of $\H_n(\A)$. Then there exists an automotphic character $\wt{\chi}$ of $\wt{\U}(\A)$ such that $\wt{\chi} \vert_{\U(\A)}=\chi$ and an irreducible globally $\wt{\chi}$-generic cuspidal representation $\wt{\pi}$ of $\G_n(\A)$ such that $\pi$ appears in $\wt{\pi} |_{\H_n(\A)}$. 
\end{prop}

\begin{proof}
Note that $\wt{\e}$ acts on $\H_n(\A)$ by a conjugation. By defining an action of $\wt{\e}\in \G_n(L)$ on $\pi$ as follows
\[(\wt{\e} \cdot \varphi)(h) \coloneqq \varphi({\wt{\e}}^{-1}h \wt{\e}), \quad \text{ for all } \varphi \in \pi, \ h \in \H_n(\A),
\]
we may view $\pi$ as a cuspidal representation of $\G_n(L) \cdot \H_n(\A)$.

Let $\Ind_{\G_n(L)\cdot \H_n(\A)}^{\G_n(\A)} (\pi)$ be the induced representation defined by the space of the functions $\Phi_{*}:\G_n(\A) \to \mathcal{H}(\pi)$, $g \mapsto \Phi_{g}$
such that
\begin{itemize}
    \item $\Phi_{hg}(x)=\Phi_g(xh)$ for $h \in \G_n(L)\cdot\H_n(\A)$, $g \in \G_n(\A)$, and $x\in \H_n(\A)$;
    \item for each $x \in \H_n(\A)$, the function $g \mapsto \Phi_g(x) \in \CC$ is a smooth function on $\G_n(\A)$.
\end{itemize}

As discussed in \cite[Remark~7.10]{AG17}, the map $\Phi_* \mapsto [g \mapsto \Phi_g(\textbf{1})]$ gives an embedding $\Ind_{\G_n(L)\cdot \H_n(\A)}^{\G_n(\A)} (\pi)$ into $\mc{A}_{cusp}(\G_n(\A))$. Choose an irreducible sub-representation  $\wt{\pi}$ in $\Ind_{\G_n(L)\cdot \H_n(\A)}^{\G_n(\A)} (\pi)$. Let $\Res\colon \mathcal{A}_{cusp}(\G_n(\A))\to \mathcal{A}_{cusp}(\H_n(\A))$, $\varphi\mapsto \varphi|_{\H_n(\A)}$ be the restriction map. Using the above embedding, we may regard $\wt{\pi}$ as an irreducible cuspidal representations of $\G_n(\A)$ such that $\Res(\wt{\pi})=\pi$.

Since $\pi$ is $\chi$-generic, there exists some $\varphi \in \pi$ such that $W_{\varphi}^{\chi}(\1)= \int_{\U(L) \bs \U(\A)} \varphi(u )\overline{\chi(u)} du\ne 0$. Choose an element $\wt{\phi} \in \wt{\pi}$ such that $\Res(\wt{\phi})=\varphi$. It can be expressed as a sum of pure tensors, namely,
$\wt{\phi}=\sum_{i=1}^\ell\wt{\phi}_i$, where each $\wt{\phi}$ is of the form $\wt{\phi}_{i}=\otimes_v \wt{\phi}_{i,v}$. Since $\Res(\wt{\phi})=\sum_{i=1}^\ell\Res(\wt{\phi}_i)=\varphi$ and $W_{\varphi}^{\chi}(\textbf{1}) \ne 0$, there is a $1\le j \le \ell$ such that $W_{\Res(\wt{\phi}_{j})}^{\chi}(\textbf{1}) \ne 0$. Let us denote this particular $\wt{\phi}_{j}$ by $\wt{\rho}$.

There is a finite set $S_0$ including all archimedean places of $L$ such that $\wt{\rho}$ is right $\big(\prod_{v \notin S_0} \mu_2(L_v)\big)\cdot \wt{\e}$-invariant. For each $v \in S_0$, decompose $\wt{\rho}_{v}=\wt{\rho}_{v,1}+\wt{\rho}_{v,2}$ (one of $\wt{\rho}_{v,i}$ might be zero) such that $\wt{\rho}_{v,i}$ is in $(-1)^{i+1}$-eigenspace of $\e_v$ in $\wt{\pi}_v$ for each $i=1,2$. Therefore, we can write $\wt{\rho}=\sum_{m=1}^k \wt{\rho}_m$ such that for each $\wt{\rho}_m=\otimes_v (\wt{\rho}_m)_v$, $(\wt{\rho}_m)_v$ is an eigen-vector of $\e_v$ for each $v \in S_0$.

Again, since $W_{\Res(\wt{\rho})}^{\chi}(\textbf{1}) \ne 0$ and $\Res(\wt{\rho})=\sum_{m=1}^k \Res(\wt{\rho}_m)$, there is a $1\le k_0 \le k$ such that $W_{\Res(\wt{\rho}_{k_0})}^{\chi}(\textbf{1}) \ne 0$. Denote $\wt{\rho}_{k_0}$ by $\wt{\varphi}_1$, which can be written as $\wt{\varphi}_1=\otimes_v \wt{\varphi}_{1,v} \in \wt{\pi}$. Let $t_0=(t_{0,v}) \in \prod_{v \in S_0}\mu_2(L_v)$ be given by \[t_{0,v} =
\begin{cases} 1, & \text{ if } \wt{\varphi}_{1,v} \text{ is an eigenvector of $\e_v$ with the eigenvalue 1}  \\ -1, & \text{ if } \wt{\varphi}_{1,v} \text{ is an eigenvector of $\e_v$ with the eigenvalue $-1$}.
\end{cases}
\]

Let $S$ be a subset of $S_0$ consisting of place $v$ such that $t_{0,v}=-1$. Since $(\wt{\e}\cdot \wt{\varphi}_1)(\textbf{1})=\wt{\varphi}_1(\textbf{1})$, we see that the number of elements in $S$ is even. 
We define a character $\wt{\chi}=\otimes_v \wt{\chi}_v$ of $\wt{\U}(\A)$ by $\wt{\chi}\vert_{\U(\A)}=\chi$ and \[\wt{\chi}(\textbf{t} \cdot \wt{\e})\coloneqq \prod_{v \in S}t_v, \quad \text{ for } \textbf{t}=(t_v) \in \mu_2(\A).\]
Since $\vert S \vert=\text{even}$, $\wt{\chi}$ is an automorphic character of $\wt{\U}(\A)$.

Then, we have 
\begin{align*}
W_{\wt{\varphi}_1}^{\wt{\chi}}(\mathbf{1})&=\int_{\mu_2(L) \bs \mu_2(\A)}\int_{\U(L) \bs \U(\A)} \wt{\varphi}_1(u(\textbf{t} \cdot \wt{\e}) )\cdot \overline{\wt{\chi}(u(\textbf{t} \cdot \wt{\e}))} du d\textbf{t}\\&=\int_{\mu_2(L) \bs \mu_2(\A)}\int_{\U(L) \bs \U(\A)} \wt{\varphi}_1(u)\cdot \overline{\wt{\chi}(u)} du d\textbf{t}\qquad(\text{since }\wt{\varphi}_1=\otimes_v\wt{\varphi}_{1,v}\text{ and } \prod_{v\in S}t_{0,v}=1)\\
&=\int_{\mu_2(L) \bs \mu_2(\A)}W_{\Res(\wt{\varphi}_1)}^\chi(\1)\,d\textbf{t}\ne0.
\end{align*} Hence, $\wt{\pi}$ is the globally $\wt{\chi}$-generic cuspidal representation of $\G_n(\A)$.
\end{proof}

\subsection{$\gamma$-factors}\label{gamma}
Let $\G$ be either $\H_n:=\SO_{2n}$ or $\J_m:=\Sp_{2m}$. Let $\pi$ be an irreducible generic representation of $\G(F)$ and $\sigma$ an irreducible generic representation of $\GL_k(F)$. Then the local twisted $\gamma$-factors $\gamma(s,\pi \times \sigma,\psi)$ are defined as the proportionality constants appearing in the functional equations of local Rankin--Selberg type integrals for $\G \times \GL_k$ or Langlands-Shahidi method. It is a product of a monomial of $t=q^{-s}$ and a rational function $\frac{Q(t)}{P(t)}$, where $Q(t),P(t) \in \CC[t]$ such that $Q(0)=P(0)=1$. For the precise definition, see \cite{Kap15} for local factors in terms of Rankin-Selberg type integrals and see \cite{Sha90, L15, L17} for local factors of Langlands-Shahidi method. In \cite{Kap15}, Kaplan demonstrated the equality of the definitions of twisted $\gamma$-factors via Rankin--Selberg integrals with those arising from the Langlands--Shahidi method as defined in \cite{Sha90}. Consequently, when $\cha(F)=0$, the choice of which definition of $\gamma$-factors to use becomes inconsequential.  In the case where $\cha(F)=p$, however, we employ the Langlands--Shahidi method to define twisted $\gamma$-factors since the definition using Rankin-Selberg integrals has not yet been established.

On the other hand, when $\G\coloneqq \GL_n$ and $\pi$ is an irreducible generic representation of $\G(F)$, $\gamma(s,\pi\times \sigma,\psi)$ is also defined as the Rankin-Selberg $\gamma$-factor of Jacquet, Piatetski-Shapiro and Shalika \cite{JSS83} (in the non-archimedean case) and  Jacquet and Shalika \cite{JacS90} (in the archimedean case) or $\gamma$-factors of Langlands-Shahidi. From the definition of $\gamma$-factors for $\GL_1 \times \GL_k$, it is easy to infer that \[\gamma(s,\chi \times \sigma,\psi)=\gamma(s,\mathbb{I} \times (\chi \otimes \sigma),\psi)\] for any character $\chi$ of $F^{\times}$ and $\sigma \in \text{Irr}(\GL_k(F))$. For brevity, we write $\gamma(s,\sigma,\psi)$ for $\gamma(s,\mathbb{I} \times \sigma,\psi)$.

It is worth mentioning that the Langlands-Shahidi method produces not only $\gamma$-factors but also local $L$-factors and local $\epsilon$-factors as follows:

Let us first consider the tempered case, that is, $\pi$ and $\sigma$ are tempered. Write $\gamma(s,\pi \times \sigma,\psi)=\alpha\cdot q^{-sk} \cdot \frac{Q(q^{-s})}{P(q^{-s})}$ for some $\alpha \in \CC^{\times}$ and polynomials $Q(t),P(t) \in \CC[t]$ such that $\text{gcd}\big(Q(t),P(t)\big)=1$ and $Q(0)=P(0)=1$. Then $L(s,\pi \times \sigma)$ and $\e(s,\pi \times \sigma, \psi)$ is defined by 
\begin{flalign}
\label{L-ftn}&L(s,\pi \times \sigma)\coloneqq Q(q^{-s})^{-1}\\
&\e(s,\pi \times \sigma, \psi) \coloneqq \gamma(s,\pi \times \sigma,\psi)\cdot \frac{L(s,\pi \times \sigma)}{L(1-s,\pi^{\vee} \times \sigma^{\vee})}.
\end{flalign}

In general, we follow the Langlands classification to define $L$-functions from Langlands-Shahidi method for any admissible generic representations $\pi$ and $\sigma$.

We recall some important properties of $\gamma$-factors of $\H_n \times \GL_k$ and $\J_m \times \GL_k$ that will be utilized later in our discussion. (Caution : the property (vi) does not exist in the $\cha(F)=p$ case.)

\begin{propt}[Properties of $\gamma$-factors in the generic case]\label{Property_gamma}
$ $ \
\begin{enumerate}
\item (Unramified twist) $\gamma(s,\pi \times \sigma|\det|_F^{s_0},\psi)=\gamma(s+s_0,\pi \times \sigma,\psi)$ $\quad \text{for } s_0 \in \RR.$
\item (Multiplicative property) Let $\P=\M\N$ be a parabolic subgroup of $\G$ such that $\M \cong \GL_{n_1}\times \cdots \times \GL_{n_r} \times \G'$, where $\G'$ is either $\SO_{2n'}$ or $\Sp_{2m'}$, which is the same type as $\G$. Let $\R= \M_{\R} \N_{\R}$ be a parabolic subgroup of $\GL_k$ such that $\M_{\R} \cong \GL_{k_1}\times \cdots \times \GL_{k_t}$.
Let $\tau_1 \otimes \cdots \otimes \tau_r \otimes \pi_0$ be an irreducible generic representation of $\M(F)$ and $\sigma_1 \otimes \cdots \otimes \sigma_t$ be an irreducible generic representation of $\M_{\R}(F)$.
Assume that $\pi$ (resp. $\sigma$) is an irreducible constituent of a parabolically induced representation $\Ind(\tau_1 \otimes \cdots \otimes \tau_r \otimes \pi_0)$ (resp. $\Ind(\sigma_1 \otimes \cdots \otimes \sigma_t)$ of $\M(F)$ (resp. $\M_{\R}(F)$.)
\[\gamma(s,\pi \times \sigma,\psi)=\gamma(s,\pi_0 \times \sigma,\psi)\prod_{i=1}^r \gamma(s,\tau_i \times \sigma,\psi)\gamma(s,\tau_i^{\vee} \times \sigma,\psi),\]
\[\gamma(s,\pi \times \sigma,\psi)=\prod_{i=1}^t \gamma(s,\pi \times \sigma_i,\psi).\]
\noindent 

\item (Dependence on $\psi$) Given $a\in F^{\times}$, denote by $\psi_{a}$ the character of $F$ given by $\psi_{a}(x) \coloneqq \psi(a x)$ for $x \in F$. Let $\omega_{\pi}, \omega_{\sigma}$ be the central character of $\pi$ and $\sigma$, respectively. Then,
\[\gamma(s,\pi \times \sigma,\psi_a)=\omega_{\pi}(a)^h \omega_{\sigma}(a)^k |a|_F^{hk(s-\frac{1}{2})} \gamma(s,\pi \times \sigma,\psi).
\]
Here, $h=2n$ if $\G=\H_n$; $h=2m+1$ if $\G=\J_m$.

\item (Unramified factors) When all data are unramified, we have \[\gamma(s,\pi \times \sigma,\psi)=\frac{L(1-s,\pi^{\vee}\times \sigma^{\vee})}{L(s,\pi \times \sigma)}.
\]

\item (Global property: Functional equation)
Let $K$ be a global field with a ring of adeles $\AA$ and $\Psi$ be a nontrivial character of $K \bs \AA$. Assume that $\Pi$ and $\Sigma$ are generic cuspidal representations of $\G(\A)$ and $\GL_k(\AA)$, respectively. Let $S$ be a finite set of places of $F$ such that for $v \notin S$, all data are unramified. Then 
\[L^S(s,\Pi \times \Sigma)=\prod_{v \in S} \gamma(s,\Pi_v \times \Sigma_v) L^S(1-s,\Pi^{\vee} \times \Sigma^{\vee}).
\]
Here, $L^S(s,\Pi \times \Sigma):= \displaystyle\prod_{v \notin S} L(s, \Pi_v \times \Sigma_v)$ is the partial $L$-function with respect to $S$.
\item (Archimedean property)
For an archimedean field $F$, 
\[\gamma(s,\pi \times \sigma,\psi)=\gamma^{\textrm{Artin}}(s,\pi \times \sigma,\psi).\]
Here, $\gamma^{\textrm{Artin}}(s,\pi \times \sigma,\psi)$ is the Artin $\gamma$-factor under the local Langlands correspondence.


\item (Tempered $L$-function)
Let $\pi,\sigma$ be irreducible tempered representations of $\G(F)$ and $\GL_k(F)$. Then $L(s,\pi \times \rho)$ is holomorphic for $\Re(s)>0$.

\item (Functorial lift of $\H_1^*$)
Suppose that $\H_1^*$ is non-split but quasi-split $\SO(V_1)$. Let $\mathbb{I}$ be the trivial character of $\H_1^*(F)$. For any character $\chi$ of $F^{\times}$, \[\gamma(s,\mathbb{I} \times \chi,\psi)=\gamma(s,\chi,\psi)\cdot \gamma(s,\chi\cdot \chi_{V_1},\psi).\]
\end{enumerate}
\end{propt}
\noindent 
The properties (i)-(vi) are proved in \cite{Sha90} for the $\text{char}(F)=0$ case and in \cite{L15, L17} for the $\text{char}(F)=p$ case. Property (vii) is proved in \cite[page 573]{CS} and \cite[Theorem~1.1]{HO13} for the $\text{char}(F)=0$ case, and in \cite[Theorem~1.1]{CGL24} for the $\text{char}(F)=p$ case. Property (viii) is proved in \cite[Proposition~5.2, \S 7.2]{CPSS11} for the $\text{char}(F)=0$ case, and in \cite[Proposition~5.1.1, \S 3.4]{Ca22} for the $\text{char}(F)=p$ case.


\begin{rem}\label{min}When $n=0$ and $m=0$, the $\gamma$-factors of $\H_0 \times \GL_k$ and $\J_0 \times \GL_k$ are defined as follows:\\
Put $\mathbb{I}_{V_0}$ (resp. $\mathbb{I}_{W_0}$) the trivial representation of $\H_0(F)$ (resp. $\J_0(F)$). Then for any irreducible generic representation $\sigma$ in $\text{Irr}(\GL_k)$,
\begin{flalign*} &\gamma(s,\mathbb{I}_{V_0} \times \sigma,\psi)\coloneqq 1 \\ &\gamma(s,\mathbb{I}_{W_0} \times \sigma,\psi)\coloneqq \gamma(s,\sigma,\psi).
\end{flalign*}
It is clear that these definitions are compatible with the local functorial lifting.
\end{rem}



\begin{rem}\label{3gamma}
When $k=1$, Lapid and Rallis \cite{LR05} defined $\gamma$-factors for (possibly non-generic) irreducible smooth representations of $\G \times \GL_1$ which arise from the doubling method. Further, they proved `Ten Commandments' which refers to the
ten properties of $\gamma$-factors.
Recently, Cai, Friedberg, and Kaplan made a breakthrough by generalizing Lapid-Rallis’s Ten Commandments to $\G\times \GL_k$ for $k\ge 1$ arising from the doubling method to the twisted doubling method in \cite{CFK}.
It is noteworthy that their $\gamma$-factors are defined not only for generic representations but also for non-generic representations of $\G(F) \times \GL_k(F)$ and they proved that three different definitions of $\gamma$-factors actually coincide for generic representations. In the proof of Theorem~\ref{thm:main}, we shall use the twisted $\gamma$-factors $\gamma(s,\pi \times \sigma,\psi)$ for (possibly) non-generic $\pi\in \text{Irr}(\J_m(F))$ and $\sigma \in \text{Irr}(\GL_k(F))$. However, due to the absence of a definition for the twisted $\gamma$-factor for non-generic representations of $\J_m(F)$ in the $\text{char}(F)=p$ case, we propose the following working hypothesis.

\begin{hyp}\label{hyp1}
The $\gamma$-factors for $\J_m \times \GL_k$ are properly defined in $\cha(F)=p$ case. Furthermore, they satisfy Property \ref{Property_gamma} (i)-(vii) (except (vi)).
\end{hyp}
\noindent When we refer the twisted $\gamma$-factor of $\J_m$ in $\cha(F)=p$ case, we work under the above working hypothesis.\end{rem}

Next, we consider the case of $\O_{2n}$ groups and define the $\gamma$-factors for $\O_{2n}\times \GL_k$. We first need the following lemma:
\begin{lem}\label{conj}For any irreducible generic representation $\pi$ of $\H_n(F)$ and $\sigma$ of $\GL_k(F)$, we have
\beq\label{tw}\gamma(s,\pi \times \sigma,\psi)=\gamma(s,\pi^{\e} \times \sigma,\psi).
\eeq
\end{lem}

\begin{proof} 
First, we prove the case when $\pi$, $\sigma$ are unique unramified quotients of principal series representations induced from minimal parabolic subgroups. There is a Borel subgroup $\B=\T\U$ of $\H_n$ and characters $\{\chi_1,\cdots,\chi_n\}$ of $F^{\times}$ such that $\pi$ is a subquotient of the induced representation $\Ind_{\B(F)}^{\H_n(F)}(\chi_1 \otimes \cdots \otimes \chi_n)$. Similarly, $\pi^{\e}$ is a subquotient of an induced representation $\Ind_{\B(F)}^{\H_n(F)}(\chi_1' \otimes \cdots \otimes \chi_n')$ for some characters $\chi_i'$ of $F^{\times}$, for $i=1, \ldots, n$. Since conjugating the induced representation by $\e$ permutes the inducing characters and their inverses, we see
 \[ \{\chi_1,\cdots,\chi_n,\chi_1^{-1},\cdots,\chi_n^{-1}\}
=\{\chi_1',\cdots,\chi_n',\chi_1'^{-1},\cdots,\chi_n'^{-1}\}.
\] 
Then the multiplicative property of $\gamma$-factors implies 
\[\gamma(s,\pi \times \sigma,\psi)=\prod_{i=1}^n\gamma(s,\chi_i \times \sigma,\psi)\gamma(s,\chi_i^{-1} \times \sigma,\psi)=\prod_{i=1}^n\gamma(s,\chi_i' \times \sigma,\psi)\gamma(s,\chi_i'^{-1} \times \sigma,\psi)=\gamma(s,\pi^{\e} \times \sigma,\psi).
\]
Next, we prove the general case. There exists a parabolic subgroup $\P=\M\N$ of $\G$ with Levi component 
\[\M=\GL_{n_1} \times \cdots \times \GL_{n_r} \times \H_{n'}
\]
where $\H_{n'}=\SO(V')$ with $V'$ a $2n'$-dimensional symmetric space over $F$, and irreducible supercuspidal representation $\tau_i$ of $\GL_{n_i}(F)$, $i=1, \ldots, r$ and $\pi'$ of $\H_{n'}(F)$ such that $\pi$ is a subquotient of $\Ind_{\P(F)}^{\H_n(F)}(\tau_1 \otimes \cdots \otimes \tau_r \otimes \pi')$. Note that
\[\Ind_{\P(F)}^{\H_n(F)}(\tau_1 \otimes \cdots \otimes \tau_r \otimes \pi')^{\e}=\begin{cases}\Ind_{\P(F)}^{\H_n(F)}(\tau_1 \otimes \cdots \otimes \tau_r \otimes (\pi')^{\e}), \quad \text{if $V' \ne 0$}\\
\Ind_{\P^{\e}(F)}^{\H_n(F)}(\tau_1 \otimes \cdots \otimes \tau_r), \quad \quad \quad \quad \ \text{if $V'=0$},
\end{cases}
\]
where $\P^{\e}$ is the $\e$-conjugate of $\P$. \par If $V'=0$, then $\P^{\e}$ is another Siegel parabolic subgroup of $\H_n$ and hence by applying the multiplicative property of $\gamma$-factors both to $\P^{\e}$ and $\P$, we have
\[\gamma(s,\pi^{\e} \times \sigma,\psi)=\prod_{i=1}^r\gamma(s,\tau_i \times \sigma,\psi)\gamma(s,\tau_i^{\vee} \times \sigma,\psi)=\gamma(s,\pi \times \sigma,\psi).
\]
Suppose that $V' \ne 0$. Then by the multiplicative property of $\gamma$-factors, we may assume that $\pi$ and $\sigma$ are supercuspidal. Now we use a standard global-to-local argument. In the characteristic zero case (resp. positive characteristic case), \cite[Proposition 5.1]{Sha90} (resp. \cite[Theorem 1.1]{GL}) implies that we have the following data:

\begin{itemize}
\item $L$, a number field (resp. global function field) such that $L_{v_0}=F$ for some finite place $v_0$ of $L$
\item $\Psi$, a nontrivial additive character of $L \bs \A$ (where $\A$ is the ring of adeles of $L$) such that $\Psi_{v_0}=\psi$
\item $\mathbb{V}$, a quadratic space over $L$ of dimension $2n$ such that $\mathbb{V}_{v_0}=V$
\item $\Pi$ and $\Sigma$, a globally generic cuspidal representations of $\SO(\VV)(\A)$ and $\GL_r(\A)$ respectively such that $\Pi_{v_0}=\tau$ and $\Sigma_{v_0}=\sigma$ and for all  places $v \ne v_0$ of $L$, $\Pi_{v}$ and $\Sigma_{v}$ are unramified. 
\end{itemize}


Using a similar argument as in the unramified case, we see that the equality (\ref{tw}) also holds for all archimedean places $v$ of $L$ because $\pi_v$ is the generic quotient of a principal series representation $I(\pi_v)$. Then Property \ref{Property_gamma}(v) and the equality (\ref{tw}) at all places $v$ except for $v_0$ imply
\[\gamma(s,\pi \times \sigma,\psi)=\gamma(s,\Pi_v \times \Sigma_v,\Psi_v)=\gamma(s,\Pi_v^{\e} \times \Sigma_v,\Psi_v)=\gamma(s,\pi^{\e} \times \sigma,\psi).\]
This completes the proof of the lemma.
\end{proof}

\begin{defn} \label{Ogamma}
For $\wt{\pi} \in \Irr(\G_n(F))$, choose an irreducible sub-representation $\pi$ of $\wt{\pi} \mid_{\H_n(F)}$, the restriction of $\wt{\pi}$ to $\H_n(F)$. Assume that $\wt{\pi}$ is $\mu_{c'}^{\pm}$-generic. Then for an arbitrary element $\sigma$  in $ \Irr(\GL_r(F))$, define
\[
\gamma(s,\wt{\pi} \times \sigma,\psi)
\coloneqq
\gamma(s,\pi\times \sigma,\psi),
\]
where $\gamma(s,\pi\times \sigma,\psi)$ is defined by Rankin-Selberg type integrals \cite{Kap15} or Langlands-Shahidi methods \cite{Sha90, L15, L17}. 
\end{defn}
By Proposition~\ref{Clifford}, if $\wt{\pi} |_{\H_n(F)}$ is reducible, it decomposes into two irreducible representations of $\H_n(F)$, which are $\e$-conjugate to each other. By Lemma~\ref{conj}, the two $\gamma$-factors associated with each component of $\wt{\pi}|_{\H_n(F)}$ are the same, and therefore, $\gamma(s,\wt{\pi} \times \sigma,\psi)$ is well-defined. With this definition, we also have
\begin{equation}\label{det}
\gamma(s,\wt{\pi} \times \sigma,\psi)=\gamma(s,(\wt{\pi}\otimes \det) \times \sigma,\psi).
\end{equation}

\begin{rem}
When $F$ is an archimedean local field, we can define $\gamma(s,\wt{\pi} \times \sigma,\psi)$ in the same way as in the non-archimedean case because the proof of Lemma~\ref{conj} also covers the archimedean case. As a result, based on Property \ref{Property_gamma}(i)--(vii) of the $\gamma$-factors for $\H_n \times \GL_k$ and Proposition~\ref{regen}, it is straightforward to verify that the $\gamma$-factors for $\G_n \times \GL_k$ also satisfy these properties. Therefore, when there is no confusion, we will use the same notation $\gamma(s,\pi \times \sigma,\psi)$ and $\gamma(s,\wt{\pi}\times \sigma,\psi)$.
\end{rem}

\subsection{Extension of Jo's result}
Now we recall a local converse theorem for $\Sp_{2m}$ obtained in \cite{Jo}.

\begin{thm}[\cite{Jo}] \label{Sp} Let $\pi$ and $\pi'$ be irreducible $\mu_1'$-generic representations of $\J_m(F)$ with the same central characters such that
\[\gamma(s,\pi \times \rho,\psi)=\gamma(s,\pi' \times \rho,\psi)\] holds for any irreducible supercuspidal representation $\rho$ of $\GL_i(F)$ with $1\le i \le m$. Then $\pi\simeq\pi'$.
\end{thm}

We extend the above result to general $\mu_{\lambda}'$-generic representations of $\J_m(F)$. 
In order to accomplish our goal, we rely on the following lemma.

\begin{lem}\label{prop:mu_d} A $\T'(F)$-orbit of irreducible $\mu_{\lambda}'$-generic representations with respect to $\psi$ is equals to a $\T'(F)$-orbit of irreducible admissible $\mu_{1}'$-generic representations with respect to $\psi_{\lambda}$.
\end{lem}

\begin{proof}
   Let $\pi$ be an irreducible $\mu_{\lambda}'$-generic representation of $\J_n(F)$ with respect to $\psi$. We aim to show that $\pi$ is $\mu_1'$-generic with respect to $\psi_{\lambda}$.

Consider the action of $t \in \T'(F)$ on a generic character $\chi'$ of $\U'(F)$, given by $(\chi')^t(u') = \chi'(t^{-1}u't)$. Specifically, for $t \in \T'(F)$ with $t(f_i) = t_if_i$ and $t(f_i^*) = t_i^{-1}f_i^*$ $(1 \leq i \leq n)$, we have
$$(\mu_{\lambda}')^t(u') = \psi\left(\sum_{i=1}^{n-1}t_i^{-1}t_{i+1}\langle u'f_i,f_{i+1}^*\rangle + \lambda t_n^{-2}\frac{\langle u'f_n^*,f_n^*\rangle}{2}\right).$$

In particular, if we choose $\underline{t} \in \T'(F)$ specifically with $t_i = \lambda^{-(n-i)}$, then we obtain
$$(\mu_{\lambda}')^{\underline{t}}(u') = \psi_{\lambda}\left(\sum_{i=1}^{n-1}\langle u'f_i,f_{i+1}^*\rangle + \frac{\langle u'f_n^*,f_n^*\rangle}{2}\right).$$

Therefore, a $(\mu_{\lambda}')^{\underline{t}}$-generic representation (with respect to $\psi$) is $\mu_1'$-generic with respect to $\psi_{\lambda}$. Since the notion of genericity is invariant under $\T'(F)$-orbits, it follows that $\pi$ is $\mu_1'$-generic with respect to $\psi_{\lambda}$.
\end{proof}
Based on Jo's result and the lemma mentioned above, we obtain the following.
\begin{thm} \label{esp} Let $\pi$ and $\pi'$ be irreducible $\mu_{\lambda}'$-generic representations of $\J_m(F)$ with the same central characters such that
\[\gamma(s,\pi \times \rho,\psi)=\gamma(s,\pi' \times \rho,\psi)\] holds for any irreducible supercuspidal representation $\rho$ of $\GL_i(F)$ with $1\le i \le m$. Then $\pi\simeq\pi'$.
\end{thm}

\section{Local theta correspondence for $(\O(V_n),\Sp(W_m)$)}\label{theta}

In this section, we introduce the local theta correspondence induced by the Weil representation of $\G_n(F) \times \J_m(F)$, denoted by $\omega_{\psi,V_n,W_m}$. Here, $\psi$ is a fixed non-trivial additive character of $F$, and $V_n$ and $W_m$ are vector spaces of dimensions $2n$ and $2m$, respectively. In what follows, we will recall some fundamental results related to the local theta correspondence.



For $\wt{\pi}\in\Irr(\G_n(F))$, the maximal $\wt{\pi}$-isotypic quotient of $\omega_{\psi,V_n,W_m}$
is of the form
\[
\wt{\pi}\boxtimes \Theta_{\psi,V_n,W_m}(\wt{\pi}),
\]
for some smooth finite length representation $\Theta_{\psi,V_n,W_m}(\wt{\pi})$ of $\J_m(F)$, called the big theta lift of $\wt{\pi}$. The maximal semisimple quotient of $\Theta_{\psi,V_n,W_m}(\wt
{\pi})$ is called the small theta lift of $\wt{\pi}$. 
Changing the role of $V_n$ and $W_m$, for $\tau\in\Irr(\J_m(F))$, 
we obtain a smooth finite length representation 
$\Theta_{\psi,W_m,V_n}(\tau)$ of $\G_n(F)$.
The maximal semi-simple quotient of $\Theta_{\psi,W_m,V_n}(\tau)$ (resp. $\Theta_{\psi,V_n,W_m}(\wt{\pi})$)
is denoted by $\theta_{\psi,W_m,V_n}(\tau)$ 
(resp. $\theta_{\psi,V_n,W_m}(\wt{\pi})$) and is called the small theta lift of $\wt{\pi}$.


The following theorem and proposition are independent of the local Langlands correspondence for $\G_n$.

\begin{thm}[Howe duality, \cite{GT1, GT2},\cite{Wa90}]\label{Howe duality}
Let $\wt{\pi}\in\Irr(\G_n(F))$ and $\tau \in \Irr(\J_m(F))$. If $\Theta_{\psi,V_n,W_n}(\wt{\pi})$ and $\Theta_{\psi,W_n,V_n}(\tau)$ are nonzero, then $\theta_{\psi,V_n,W_m}(\wt{\pi})$ and $\theta_{\psi,W_m,V_n}(\tau)$ are irreducible. Moreover, for $\wt{\pi_1},\wt{\pi_2} \in \Irr(\G_n(F))$ which occur as quotients of $\omega_{\psi,V_n,W_m}$, if $\theta_{\psi,V_n,W_m}(\wt{\pi_1})\simeq \theta_{\psi,V_n,W_m}(\wt{\pi_2})$, then $\wt{\pi_1} \simeq \wt{\pi_2}$.
\end{thm}

\begin{prop}\label{tem0} Let $\wt{\pi} \in \Irr_{\temp}(\G_n(F))$. If $\Theta_{\psi,V_n,W_n}(\wt{\pi})$ is non-zero, then $\Theta_{\psi,V_n,W_n}(\wt{\pi})$ is an irreducible tempered representation of $\J_n(F)$.
\end{prop}
\begin{proof} In $\cha(F)=0$, it is a part of \cite[Proposition  C.4]{GI1}. The key ingredients of the proof therein are the Howe duality and the Kudla's filtrartion on the normalized Jacquet module of the Weil representation. However, Howe duality does hold for $\cha(F)\ne 2$ case and the Kudla's computation on the Jacquet module of the Weil representation also holds for $\cha(F) \ne 2$ (see \cite[III.8]{Ku}). Except for these, other arguments in \cite[Theorem C.4]{GI1} equally apply to $\cha(F) =p$ case too.
\end{proof}



Now we can prove the following proposition. In the case of a split $\J_n(F)$, it has been proved in \cite[Corollary~2.5]{GRS} that when $\tau\in\textrm{Irr}(\J_n(F))$ is $\mu_\lambda'$-generic and $\Theta_{\psi,W_n,V_n}(\tau)$ is nonzero, then $\theta_{\psi,W_n,V_n}(\tau)$ is $(\mu_{-\lambda})^{-1}$-generic.

\begin{prop}\label{Theta1} Let $\wt{\pi} \in \Irr_{\temp}(\G_n(F))$.
\begin{enumerate}
\item If $\wt{\pi}$ is $\mu_{c'}^{+}$-generic, then $\Theta_{\psi,V_n,W_n}(\wt{\pi})$ is nonzero and $\theta_{\psi,V_n,W_n}(\wt{\pi})$ is $(\mu_{-c'}')^{-1}$-generic. 
\item If $\wt{\pi}$ is $\mu_{c'}^{-}$-generic but not $\mu_{c'}^{+}$-generic, then $\Theta_{\psi,V_n,W_n}(\wt{\pi})$ is zero or $\theta_{\psi,V_n,W_n}(\wt{\pi})$ is not $(\mu_{-c'}')^{-1}$-generic. 
\end{enumerate}
\end{prop}

\begin{proof} 
Write $(\omega_{\psi,V_n,W_n})_{\U'(F),(\mu_{-c'}')^{-1}}$ for the twisted Jacquet module of $\omega_{\psi,V_n,W_n}$ with respect to $\U'(F)$ and $(\mu_{-c'}')^{-1}$ (i.e. the quotient space $\omega_{\psi,V_n,W_n} / \mc{V}$, where $\mc{V}$ is a subspace  spanned by $\{\omega_{\psi,V_n,W_n}(u)\cdot \phi-(\mu_{-c'}')^{-1}(u)\cdot \phi\}_{u\in \U', \phi \in \omega_{\psi,V_n,W_n}}$.) By Theorem~\ref{thm:tjw1}, we have
\begin{align*}\Hom_{\G_n(F)\times \U'(F)}\big(\omega_{\psi,V_n,W_n}, \ \wt{\pi} \otimes (\mu_{-c'}')^{-1}\big) \cong \Hom_{\G_n(F)}\big((\omega_{\psi,V_n,W_n})_{\U'(F),(\mu_{-c'}')^{-1}}, \ \wt{\pi}\big)\cong \Hom_{\G_n(F)}\big(\ind_{\wt{\U}(F)}^{\G_n(F)} (\mu_{c'}^{+}), \ \wt{\pi} \big)\\
\cong \Hom_{\wt{\U}(F)}\big(\mu_{c'}^{+}, \ (\wt{\pi}^{\vee}\mid_{\wt{\U}(F)})^{\vee} \big) \cong \Hom_{\wt{\U}(F)}\big(\wt{\pi}^{\vee}, \ (\mu_{c'}^{+})^{\vee} \big)\cong \Hom_{\wt{\U}(F)}\big(\wt{\pi}, \ \mu_{c'}^{+} \big),
\end{align*}
where the last equality follows from the facts that $\wt{\pi}$ and $\mu_{c'}^{+}$ are unitary.\par 
On the other hand, 
\[\Hom_{\G_n(F) \times \U'(F)}\big(\omega_{\psi,V_n,W_n}, \ \wt{\pi} \otimes (\mu_{c-'}')^{-1}\big)\cong \Hom_{\U'(F)}\big(\Theta_{\psi,V_n,W_n}(\wt{\pi}), \  (\mu_{-c'}')^{-1}\big).
\]
If $\wt{\pi}$ is $\mu_{c'}^{+}$-generic, then $\Hom_{\wt{\U}(F)}\big(\wt{\pi}, \ \mu_{c'}^{+} \big) \ne 0$, and henceforth, 
$\Hom_{\U'(F)}\big(\Theta_{\psi,V_n,W_n}(\wt{\pi}), \  (\mu_{-c'}')^{-1}\big) \ne 0$. Therefore, we have $\Theta_{\psi,V_n,W_n}(\wt{\pi})\ne 0$, and by Proposition~\ref{tem0},  $\Hom_{\U'(F)}\big(\theta_{\psi,V_n,W_n}(\wt{\pi}), \  (\mu_{-c'}')^{-1}\big) \ne 0$. This proves (i).\par
If $\wt{\pi}$ is not $\mu_{c'}^{+}$-generic, then $\Hom_{\wt{\U}(F)}\big(\wt{\pi}, \ \mu_{c'}^{+} \big)=0$, and from the above argument we have \[\Hom_{\U'(F)}\big(\Theta_{\psi,V_n,W_n}(\wt{\pi}), (\mu_{-c'}')^{-1}\big)=0\]
which proves (ii).
\end{proof}

\begin{rem} Proposition~\ref{Theta1} can be obtained using the local Langlands correspondence for $\G_n$ (\cite{AG17}) and $\J_n$ (\cite{Ar}). However, since our other goal is to prove our main theorems without using Arthur's results, we prove it by computing the twisted Jacquet module of Weil representation, which is independent of Arthur's results.
\end{rem}

The following is an easy consequence of Lemma~\ref{gen} and Proposition~\ref{Theta1}.
\begin{cor}[{\cite[Cor~9.3]{MS20}}] Let $\pi \in \Irr_{\temp}(\H_n(F))$. If $\pi$ is $\mu_{c'}$-generic, then there is a unique $\mu_{c'}^{+}$-generic irreducible constituent $\wt{\pi}$ of $\Ind_{\H_n(F)}^{\G_n(F)}(\pi)$ such that $\Theta_{\psi,V_n,W_n}(\wt{\pi})$ is nonzero and $\theta_{\psi,V_n,W_n}(\wt{\pi})$ is $(\mu_{-c'}')^{-1}$-generic. 
\end{cor}


\section{$\gamma$-factors and the local theta correspondence}
\label{gammacf}

In this section, we examine the relationship of $\gamma$-factors under the local theta correspondence for $(\O_{2n},\Sp_{2m})$.

\noindent Let $\wt{\pi} \in \Irr(\G_n(F))$, $\tau \in \Irr(\J_m(F))$. Suppose that they correspond to each other under the local theta correspondence for $(\G_n,\J_m)$. 
In this section, we establish a precise relation between the two twisted $\gamma$-factors of $\wt{\pi}$ and $\tau$.
Consider a unitary induced representation 
\begin{equation}\label{eq4_1}
\Ind_{\P(F)}^{\G_n(F)}
(\rho_1|\det|_F^{s_1}\otimes\cdots\otimes \rho_r|\det|_F^{s_r}\otimes \wt{\pi_0}),
\end{equation}
where
\begin{itemize}
\item
$V=V_{n}$ and $V_0 = V_{n_0}$ are symmetric spaces of dimension $2n$ and $2n_0$, 
respectively;
\item
$\P$ is a parabolic subgroup of $\G_n$ with the Levi subgroup isomorphic to
$\GL_{n_1}\times\cdots\times\GL_{n_r}
\times \G_{n_0}$; 
\item
$\rho_i$ is an irreducible unitary supercuspidal representation of $\GL_{n_i}(F)$;
\item
$s_i$ is a real number such that $s_1 \geq \dots \geq s_r \ge 0$; 
\item
$\wt{\pi_0}$ is an irreducible $\mu_{c'}^{\pm}$-generic tempered representation of $\G_{n_0}(F)$.
\end{itemize}

It is well-known and can be easily verified from \cite[Theorem 4.2]{BJ01} and \cite[Lemma 4.11]{Jo} that an irreducible $\mu_{c'}^{\pm}$-generic representation $\wt{\pi}$ of $\G_n(F)$ is a subquotient of the form (\ref{eq4_1}) (See \cite{BJ01} and \cite[Lemma 4.11]{Jo}).
Then we say that $\wt{\pi}$ has a tempered support $(P, \rho_1,\cdots,\rho_r;\wt{\pi_0})$ with exponents $(s_1,\cdots,s_r)$. Especially, when $\wt{\pi_0}$ is supercuspidal, we say that $\wt{\pi}$ has supercuspidal support $(P, \rho_1,\cdots,\rho_r;\wt{\pi_0})$ with exponents $(s_1,\cdots,s_r)$. Similarly, we can define the supercuspidal support with exponents of an irreducible $\mu_{c'}'$-generic representation of $\J_m(F)$. The notion of supercuspidal support can also be defined without exponents by allowing the inducing representations $\rho_i$ to be supercuspidal (not necessarily unitary). It is also well-known that supercuspidal supports are uniquely determined up to conjugacy class of Weyl group elements. We note that the central characters and $\mu_{c'}^{\pm}$-genericity of $\wt{\pi}$ and $\wt{\pi_0}$ are the same.

The following is a summary of Kudla's supercuspidal support theorem.

\begin{prop}[{\cite[Theorem 7.1]{Ku}} or {\cite[Proposition 5.2]{GI1}}]\label{ksc} 
Let $\wt{\pi} \in \Irr(\G_n(F))$ be such that $\Theta_{\psi,V_n,W_m}(\wt{\pi})$ is nonzero. Put $\tau=\theta_{\psi,V_n,W_m}(\wt{\pi})$. We assume that $\wt{\pi} \in \Irr(\G_n(F))$ and $\tau \in \Irr(\J_m(F))$ have supercuspidal supports $(\rho_1,\cdots,\rho_r;\wt{\pi_0})$ with $\wt{\pi_0} \in \Irr(\G_{n_0}(F))$ and $(\rho_1',\cdots,\rho_s';\tau_0)$ with $\tau_0 \in \Irr(\J_{m_0}(F))$, respectively. Put $l=2n-2m-1$ and $l_0=2n_0-2m_0-1$. Then the following holds:
\begin{enumerate}
\item $\tau_0=\theta_{\psi,V_0,W_0}(\wt{\pi_0})$.
\item If $ m-m_0 \le n-n_0$, then
\[\{\rho_1,\cdots,\rho_r\}=\{ |\cdot|^{\frac{l-1}{2}},|\cdot|^{\frac{l-3}{2}},\cdots, |\cdot|^{\frac{l_0+1}{2}},\rho_1' \chi_{V_n}^{-1},\cdots,\rho_s' \chi_{V_n}^{-1}  \}.\]
\item If $ m-m_0 \ge n-n_0$, then
\[ \{\rho_1',\cdots,\rho_s'\}=\{\chi_{V_n}|\cdot|^{\frac{-(l+1)}{2}},\chi_{V_n}|\cdot|^{\frac{-(l+3)}{2}},\cdots,  \chi_{V_n}|\cdot|^{\frac{-(l_0-1)}{2}},\rho_1\chi_{V_n},\cdots,\rho_r\chi_{V_n}  \}.\]
\end{enumerate}
\end{prop}

For an irreducible smooth representation $\rho$ of $\GL_r(F)$ and a $2n$-dimensional symmetric space $V$ over $F$, set 
\[\gamma(\sigma,V)\coloneqq \begin{cases}\frac{\gamma(s,\sigma,\psi)}{\gamma(s,\sigma\chi_{V},\psi)},&\quad \text{ if } \dim(V_{\an})=0 \\ 1, &\quad \text{ if } \dim(V_{\an})=2. \end{cases}\]

The following lemma is an application of Proposition~\ref{ksc}, which illustrates the precise relationship between the $\GL_1$-twisted $\gamma$-factors of $\wt{\pi}$ and $\theta_{\psi,V_n,W_m}(\wt{\pi})$ when $\wt{\pi}$ is an unramified representation. When $\cha(F)=0$, it is a part of \cite[Lemma 11.8]{GI1}. However, the formula therein is only valid when $\G_n$ is non-split but quasi-split. For the case when $\G_n$ is split, some adjustments are required.

\begin{lem}\label{Theta0} Let $\wt{\pi} \in \Irr(\G_n(F))$. Suppose that $\wt{\pi}$ is unramified and $\Theta_{\psi,V_n,W_m}(\wt{\pi})$ is nonzero. Put $\tau=\theta_{\psi,V_n,W_m}(\wt{\pi})$. Write $l=2n-2m-1$. Then for any character $\sigma$ of $\GL_1(F)$, the following hold:
\begin{enumerate}
\item if $l \ge 1$, then
\[\gamma(s,\tau \times \sigma ,\psi)=\gamma(\sigma,V_n) \cdot \gamma(s,\wt{\pi}\times  \sigma \chi_{V_n},\psi) \cdot \prod_{i=1}^{l}  \gamma(s+\frac{l+1}{2}-i,\sigma \chi_{V_n},\psi)^{-1}.\]

\item if $l \le -1$, then
\[\gamma(s,\tau \times \sigma,\psi)=\gamma(\sigma,V_n) \cdot \gamma(s,\wt{\pi} \times \sigma \chi_{V_n},\psi)\cdot \prod_{i=1}^{-l}\gamma(s+\frac{-l+1}{2}-i,\sigma \chi_{V_n},\psi).\]
\end{enumerate}
\end{lem}

\begin{proof}
We keep the notation in Proposition~\ref{ksc}, i.e., $\wt{\pi}$ and $\tau$ have supercuspidal supports $(\rho_1,\cdots,\rho_r;\wt{\pi_0})$ with $\wt{\pi_0} \in \Irr(\G_{n_0}(F))$ and $(\rho_1',\cdots,\rho_s';\tau_0)$ with $\tau_0 \in \Irr(\J_{m_0}(F))$, respectively. For convenience in notation, set $\chi=\chi_{V_n}$. Note that $\chi=\chi^{-1}$. Since $V_{n_0}$ (resp. $W_{m_0}$) is the anisotropic kernel of $V_n$ (resp. $W_m$), $n_0=0 \text{ or }1$ (resp. $m_0=0$). Furthermore, $\rho_i,\rho_j'$ are 1-dimensional character of $\GL_1(F)$ and $\wt{\pi_0}=\mathbb{I}_{V_{n_0}}$ (resp. $\tau_0=\mathbb{I}_{W_{m_0}}$) the trivial representation of $\G_{n_0}(F)$ (resp. $\J_{m_0}(F)$).  

We first prove (i). By Proposition~{\ref{ksc}} (ii),
\begin{flalign*}
\gamma(s,\wt{\pi} \times \sigma\chi,\psi)
&
=\gamma(s,\mathbb{I}_{V_{n_0}} \times \sigma\chi,\psi) \cdot \prod_{j=1}^{s} \big(\gamma(s,\rho_i' \chi^{-1} \times \sigma\chi,\psi) \cdot \gamma(s,(\rho_i' \chi^{-1})^{-1} \times \sigma\chi,\psi)\big)\\
&
\times \prod_{i=1}^{\frac{l-l_0}{2}} \gamma(s+\frac{l+1}{2}-i,\sigma\chi,\psi) \cdot \gamma(s+i-\frac{l+1}{2},\sigma\chi,\psi) \\
&
=\frac{\gamma(s,\mathbb{I}_{V_{n_0}} \times \sigma\chi,\psi)}{\gamma(s,\mathbb{I}_{W_{m_0}}\times \sigma,\psi)}
\cdot 
\gamma(s,\mathbb{I}_{W_{m_0}}\times \sigma ,\psi)
\prod_{j=1}^{s}\gamma(s,\rho_i' \times \sigma,\psi)\cdot \gamma(s,(\rho_i' )^{-1} \times \sigma,\psi)\\
&
\times \frac{\prod_{i=1}^{l}  \gamma(s+\frac{l+1}{2}-i,\sigma\chi,\psi)}{\prod_{i=1}^{l_0}  \gamma(s+\frac{l_0+1}{2}-i,\sigma\chi,\psi) }\\ 
&
= \gamma(s,\tau \times \sigma,\psi) \cdot \frac{\gamma(s,\mathbb{I}_{V_{n_0}} \times \sigma\chi,\psi)}{\gamma(s,\mathbb{I}_{W_{m_0}}\times \sigma,\psi)} \cdot \frac{\prod_{i=1}^{l}  \gamma(s+\frac{l+1}{2}-i,\sigma\chi,\psi) }{\prod_{i=1}^{l_0}  \gamma(s+\frac{l_0+1}{2}-i,\sigma\chi,\psi)}.\end{flalign*}
Here, $\prod_{i=1}^{l_0}  \gamma(s+\frac{l_0+1}{2}-i,\sigma\chi,\psi)$ is defined as  $\gamma(s,\sigma\chi,\psi)^{-1}$ if $l_0=-1$.

\noindent If $n_0=0$, then $l_0=-1$. By Remark~{\ref{min}}, we have
\[
\gamma(s,\wt{\pi} \times \sigma\chi,\psi) = \gamma(s,\tau \times \sigma,\psi) \cdot \frac{\gamma(s,\sigma\chi,\psi)}{\gamma(s,\sigma,\psi)} \cdot \prod_{i=1}^{l}  \gamma(s+\frac{l+1}{2}-i,\sigma\chi,\psi).
\]
If $n_0=1$, then $l_0=1$. By Property~{\ref{Property_gamma}} (viii) and Remark~{\ref{min}}, we have
\begin{flalign*}
&
\gamma(s,\wt{\pi} \times \sigma\chi,\psi) = \gamma(s,\tau \times \sigma,\psi) \cdot \frac{\gamma(s,\sigma\chi,\psi)\cdot \gamma(s,\sigma,\psi)}{\gamma(s,\sigma,\psi)\cdot \gamma(s, \sigma\chi,\psi)} \cdot \prod_{i=1}^{l}  \gamma(s+\frac{l+1}{2}-i,\sigma\chi,\psi) \\ 
&
=\gamma(s,\tau \times \sigma,\psi) \cdot \prod_{i=1}^{l} \gamma(s+\frac{l+1}{2}-i,\sigma\chi,\psi).\end{flalign*}

Next, we prove (ii). By Proposition~{\ref{ksc}} (iii), 
\begin{flalign*}
\gamma(s,\tau \times \sigma,\psi)
&
=\gamma(s,\mathbb{I}_{W_{m_0}} \times \sigma,\psi)\cdot \prod_{j=1}^{r} \gamma(s,\rho_j \chi \times \sigma,\psi)\cdot  \gamma(s,(\rho_j \chi)^{-1} \times \sigma,\psi)\\
&
\times \prod_{i=1}^{\frac{l_0-l}{2}} \gamma(s-\frac{(l-1)}{2}-i,\sigma\chi,\psi) \cdot \gamma(s+i+\frac{l-1}{2},\sigma\chi^{-1},\psi)\\
&
=\frac{\gamma(s,\mathbb{I}_{W_{m_0}} \times \sigma,\psi)}{\gamma(s,\mathbb{I}_{V_{n_0}}\times \sigma\chi ,\psi)}
\cdot
\gamma(s,\mathbb{I}_{V_{n_0}}\times \sigma \chi,\psi) \cdot
\prod_{j=1}^{r} \gamma(s,\rho_j \times \sigma\chi ,\psi)\cdot  \gamma(s,(\rho_j )^{-1} \times \sigma\chi ,\psi)\\
&
\times \frac{ \prod_{i=1}^{-l} \gamma(s-\frac{(l-1)}{2}-i,\sigma\chi,\psi)}{\prod_{i=1}^{-l_0} \gamma(s-\frac{(l_0-1)}{2}-i,\sigma\chi,\psi)}\\ 
&
=\gamma(s,\wt{\pi} \times \sigma\chi,\psi) \cdot \frac{\gamma(s,\mathbb{I}_{W_{m_0}} \times \sigma,\psi)}{\gamma(s,\mathbb{I}_{V_{n_0}}\times \sigma\chi,\psi)} \cdot \frac{\prod_{i=1}^{-l} \gamma(s-\frac{(l-1)}{2}-i,\sigma\chi,\psi) }{\prod_{i=1}^{-l_0} \gamma(s-\frac{(l_0-1)}{2}-i,\sigma\chi,\psi)}.\end{flalign*}
Here, $\prod_{i=1}^{-l_0} \gamma(s-\frac{(l_0-1)}{2}-i,\sigma\chi,\psi)$ is defined as  $\gamma(s,\sigma\chi,\psi)^{-1}$ if $l_0=1$.

\noindent If $n_0=0$, then $l_0=-1$. By Remark~{\ref{min}},
\[
\gamma(s,\tau \times \sigma,\psi)=\gamma(s,\wt{\pi} \times \sigma\chi,\psi)\cdot\frac{\gamma(s,\sigma,\psi)}{\gamma(s, \sigma\chi,\psi)} \cdot \prod_{i=1}^{-l} \gamma(s-\frac{(l-1)}{2}-i,\sigma\chi,\psi).\]
If $n_0=1$, then $l_0=1$.  By Property~{\ref{Property_gamma}} (viii) and Remark~{\ref{min}},  
\begin{flalign*}
&
\gamma(s,\tau \times \sigma,\psi)
=\gamma(s,\wt{\pi} \times \sigma\chi,\psi)\cdot \frac{\gamma(s,\sigma,\psi)}{\gamma(s,\sigma,\psi) \cdot \gamma(s,\sigma\chi,\psi) \cdot \gamma(s,\sigma\chi,\psi)^{-1}} \cdot \prod_{i=1}^{-l} \gamma(s-\frac{(l-1)}{2}-i,\sigma\chi,\psi)\\ 
&
=\gamma(s,\wt{\pi} \times \sigma\chi,\psi) \cdot  \prod_{i=1}^{-l} \gamma(s-\frac{(l-1)}{2}-i,\sigma\chi,\psi).
\end{flalign*} 
This completes the proof.
\end{proof}

The following theorem is the generalization of Lemma~{\ref{Theta0}}.

\begin{thm}\label{Theta2} Let $\wt{\pi} \in \Irr(\G_n(F))$. Suppose that $\wt{\pi}$ is $\mu_{c'}^{\pm}$-generic and $\Theta_{\psi,V_n,W_m}(\wt{\pi})$ is nonzero. Put $\tau=\theta_{\psi,V_n,W_m}(\wt{\pi})$. Write $l=2n-2m-1$. Then for any irreducible generic representation $\sigma$ of $\GL_r(F)$, the following hold:
\begin{enumerate}
\item if $l \ge 1$, then
\[\gamma(s,\tau \times \sigma ,\psi)=\gamma(\sigma,V_n) \cdot \gamma(s,\wt{\pi}\times  \sigma \chi_{V_n},\psi) \cdot \prod_{i=1}^{l}  \gamma(s+\frac{l+1}{2}-i,\sigma \chi_{V_n},\psi)^{-1}.\]

\item if $l \le -1$, then
\[\gamma(s,\tau \times \sigma,\psi)=\gamma(\sigma,V_n) \cdot \gamma(s,\wt{\pi} \times \sigma \chi_{V_n},\psi)\cdot \prod_{i=1}^{-l}\gamma(s+\frac{-l+1}{2}-i,\sigma \chi_{V_n},\psi).\]
\end{enumerate}
\end{thm}

\begin{rem}
Theorem \ref{Theta2} can be considered as a generalization of \cite[Theorem 11.5]{GI1}. It is important to note that \cite[Theorem ~11.5]{GI1} utilizes Lapid-Rallis's $\gamma$-factors for $\G_n \times \GL_1$ and $\J_m \times \GL_1$ defined in  \cite{LR05}. Lapid-Rallis's $\gamma$-factors are defined for all smooth representations, not necessarily limited to generic ones. In Proposition~\ref{Theta1}, we observed that $\tau=\theta_{\psi,V_n,W_m}(\widetilde{\pi})$ is generic if $n=m$. However, when $n\neq m$, $\tau$ need not necessarily be a generic representation. Hence, for $\gamma(s,\tau \times \sigma, \psi)$ that appears in Theorem~\ref{Theta2}, we adopt the definition of $\gamma$-factors for $\J_m \times \GL_k$ in \cite{CFK} for $\cha(F)=0$ case and in \textbf{Working hypothesis} on $\cha(F)=p$ case. See Remark~\ref{3gamma}.
\end{rem}


To begin the proof of Theorem~\ref{Theta2}, we state the following lemma which will be used in our argument. It directly follows from the Kudla’s supercuspidal support theorem, i.e., Proposition~\ref{ksc}.
\begin{lem} \label{con}Let $\{\wt{W_r}\}$ be the Witt tower containing $W_m$, and let $\wt{\pi} \in \Irr(\G_n(F))$. Suppose that Proposition~\ref{Theta2} holds for some $W' \in \{\wt{W_r}\}$ such that $\Theta_{\psi,V_n,W'}(\wt{\pi}) \ne 0$. Then Theorem~\ref{Theta2} holds for all $W$ in $\{\wt{W_r}\}$ such that  $\Theta_{\psi,V_n,W}(\tau) \ne 0$.
\end{lem}
Now we are ready to prove Theorem~\ref{Theta2}.

\begin{proof}[Proof of Theorem~\ref{Theta2}] 
The proof follows a similar line of reasoning in \cite[Theorem 11.5]{GI1}. Write $V=V_n$.


By the multiplicative property of $\gamma$-factors, it is enough to consider the case when $\wt{\pi}$ and $\sigma$ are supercuspidal. This can be proved using a global-to-local argument. Let $\pi$ be an irreducible supercuspidal constituent of $\wt{\pi}\vert_{\SO(V)}$. (See \cite[Section 2]{BJ01} for supecuspidality of $\pi$.)
Suppose we have the following data:
\begin{itemize}
\item $L$, a totally imaginary number field (or a global function field) with $\A$ its adele ring such that $L_{v_0}=F$ for some finite place $v_0$ of $L$
\item $\Psi$, a nontrivial additive character of $L \bs \A$ such that $\Psi_{v_0}=\psi$
\item $\mf{U}$, a generic automorphic character of $U(L) \bs U(\A)$ associated to $\Psi$ such that $\mf{U}_{v_0}=\mu_{c'}$
\item $\mathbb{\WW}$, a symplectic space over $L$ of dimension $2m$, with associated isometry group $\Sp(\WW)$
\item $\mathbb{V}$, a symmetric space over $L$ of dimension $2n$ such that $\mathbb{V}_{v_0}=V$, with associated isometry group $\O(\VV)$ and Hecke character $\chi$ of $\AA^{\times}$ such that $\chi_{v_0}=\chi_{V}$
\item $\{\mathbb{W}_k\}$, the tower of symplectic spaces over $L$ containing $\WW$ 
\item $S$, a finite set consisting all archimedean places of $L$ and $v_0$ (if $L$ is a global function field, $S=\{v_0\}$) 
\item $\Pi$, a globally $\mf{U}$-generic cuspidal representation of $\SO(\VV)(\AA)$ such that $\Pi_{v_0}=\pi$ and for all places $v \notin S$ of $L$, $\Pi_v$ is unramified
\item $\Sigma$, a globally generic cuspidal representation of $\GL_r(\A)$ such that $\Sigma_{v_0}=\sigma$ and for all places $v \notin S$ of $L$, $\Sigma_{v}$ is unramified
\item $k_0$, the first index of the global theta lift tower 
$\{\Theta_{\VV,\WW_k}(\Pi)\}$ with respect to $\Psi$ and $\chi$
such that $\Theta_{\VV,\WW_{k_0}}(\Pi)$ is nonzero
\item $\Xi$, an irreducible constituent of $\Theta_{\VV,\WW_{k_0}}(\Pi)$
\end{itemize}
The existence of such $L$ follows from a similar argument in \cite[Lemma~{5.2}]{MS20}. We can always construct such $\Pi$ and $\Sigma$. (For the $\cha(F)=0$ case, see \cite[Proposition 5.1]{Sha90} and for the $\cha(F)=p$ case, see \cite[Theorem~1.1]{GL}.) Furthermore, Theorem \ref{a1} tells that such $k_0$ exists and $\Xi$ is cuspidal. By Proposition~\ref{exgen}, there is an automorphic generic character $\wt{\mf{U}}$ of $\wt{\U}(\A)$ such that $\wt{\mf{U}}\vert_{\U(\A)}=\mf{U}$ and a globally $\wt{\mf{U}}$-generic cuspidal representation $\wt{\Pi}$ of $\O(\VV)(\AA)$ such that $\Res(\wt{\Pi})=\Pi$. By Proposition~\ref{Clifford}, we have $\wt{\Pi}_{v_0}=\wt{\pi} \text{ or } \wt{\pi}\otimes \det$.

Write $l_0=2n-2k_0-1$. For simplicity, we assume $l_0\ge1$. (The case $l_0 \le -1$ can be proved similarly.) For unramified places $v \notin S$, an unramified representation $\Sigma_{v}$ is considered the unique unramified quotient of a principal series $\Ind_{\mathbb{B}(L_v)}^{\GL_r(L_v)}
(\chi_1\otimes\cdots\otimes \chi_r)$, where $\chi_i$ are unramified character of $F^{\times}$ and $\mathbb{B}$ is a Borel subgroup of $\GL_r$. Therefore, multiplicativity property of $\gamma$-factors together with Lemma~{\ref{Theta0}} imply the following:
\begin{equation}\label{gamma:iden}
\gamma(s,\Xi_{v} \times \Sigma_{v},\Psi_v)=\gamma(\Sigma_v,\mathbb{V}_{v})\cdot\gamma(s,\wt{\Pi}_v \times \Sigma_{v} \chi_{v},\Psi_v)\cdot \prod_{i=1}^{l_0}\gamma(s+\frac{l_0+1}{2}-i,\Sigma_{v} \chi_{v},\Psi_v)^{-1}.
\end{equation}
Note that for all places $v \notin S$, $\Xi_v, \Pi_v$ and $\Sigma_v$ are unramified. Therefore, Property~\ref{Property_gamma} (iv) of $\gamma$-factors implies that by taking products of (\ref{gamma:iden}) over all places $v \notin S$ we have
\[\frac{L^S(1-s,\Xi^{\vee} \times \Sigma^{\vee})}{L^S(s,\Xi \times \Sigma)}=\gamma^S(\Sigma_v,\mathbb{V}_{v}) \cdot \frac{L^S(1-s,\wt{\Pi}^{\vee} \times \Sigma^{\vee} \chi^{\vee})}{L^S(s,\wt{\Pi} \times \Sigma \chi)} \times \prod_{i=1}^{l_0} \frac{L^S(1-s-\frac{l_0+1}{2}+i,\Sigma^{\vee} \chi^{-1})^{-1}}{L^S(s+\frac{l_0+1}{2}-i,\Sigma \chi)^{-1}}.\]
For any archimedean place $v$ of $L$, $L_v = \CC$. The theta correspondence for complex groups is well-understood and can be described in terms of the local Langlands correspondence (see \cite{AB}). Applying the archimedean property of $\gamma$-factors, we have the identity (\ref{gamma:iden}) at the archimedean places $v$ of $L$ as well.

By the global functional equation of $\gamma$-factors,
\[
\frac{L^S(1-s,\Xi^{\vee} \times \Sigma^{\vee})}{L^S(s,\Xi \times \Sigma)} \times \prod_{v \in S}\gamma(s,\Xi_{v} \times \Sigma_{v},\Psi_{v})
=1
\]
\[
=\gamma^S(\Sigma_v,\mathbb{V}_{v}) \cdot \frac{L^S(1-s,\wt{\Pi}^{\vee} \times \Sigma^{\vee} \chi^{\vee})}{L^S(s,\wt{\Pi} \times \Sigma \chi)} \times \prod_{i=1}^{l_0} \frac{L^S(1-s-\frac{l_0+1}{2}+i,\Sigma^{\vee} \chi^{-1})^{-1}}{L^S(s+\frac{l_0+1}{2}-i,\Sigma \chi)^{-1}} 
\]
\[ \times \prod_{v \in S} \left( \gamma(\Sigma_v,\mathbb{V}_{v}) \cdot \gamma(s,\wt{\Pi}_{v}\times \Sigma_{v} \chi_{v},\Psi_{v}) 
\times  \prod_{i=1}^{l_0}\gamma(s+\frac{l_0+1}{2}-i,\Sigma_{v} \chi_{v},\Psi_{v})^{-1} \right)
.\]
By canceling the identities at places outside $v_0$, we have
\[
\gamma(s,\Xi_{v_0} \times \Sigma_{v_0},\Psi_{v_0}) 
=\gamma(\Sigma_{v_0},\mathbb{V}_{v_0}) \cdot \gamma(s,\wt{\Pi}_{v_0}\times \Sigma_{v_0} \chi_{v_0},\Psi_{v_0})\cdot   \prod_{i=1}^{l_0}\gamma(s+\frac{l_0+1}{2}-i,\Sigma_{v_0} \chi_{v_0},\Psi_{v_0})^{-1}.
\]
Since $\wt{\Pi}_{v_0}=\wt{\pi} \text{ or } \wt{\pi} \otimes \det$, $\Sigma_{v_0}=\sigma$, $\mathbb{V}_{v_0}=V$ and $\Xi_{v_0}=\theta_{\psi,\VV_{v_0},(\WW_{k})_{v_0}}(\wt{\pi})$, (\ref{det}) implies the desired identity for this specific theta lift $\Xi_{v_0}$. Moreover, Lemma~\ref{con} establishes the desired identity for any arbitrary non-zero theta lift.
\end{proof}

The following corollary plays a key role in proving our main theorem.
\begin{cor}\label{Theta3} Let $\wt{\pi} \in \Irr_{\rm{temp}}(\G_n(F))$ and suppose that $\wt{\pi}$ is $\mu_{c'}^{+}$-generic. Put $\tau=\theta_{\psi,V_n,W_n}(\wt{\pi}) \in \Irr_{\temp}(\J_n(F))$. Then for any irreducible supercuspidal representation $\sigma$ of $\GL_r(F)$,
   \[\gamma(s,\tau \times \sigma,\psi)=\gamma(\sigma,V_n)\cdot \gamma(s,\wt{\pi} \times \sigma \chi_{V_n},\psi)\cdot\gamma(s,\sigma \chi_{V_n},\psi).\]
\end{cor}

\begin{proof} Note that $\tau$ is nonzero and tempered by Proposition~\ref{tem0} and Proposition~\ref{Theta1}. Then this is a special case of Theorem~\ref{Theta2} when $n=m$, $l=-1$, and $\sigma$ is a supercuspidal representation.
\end{proof}

\section{The proof}\label{The proof}
In this section, we prove our main theorem. The proof consists of two steps; the first step is to prove the following tempered case.

\begin{thm}\label{tem}Let $\wt{\pi_1}$ and $\wt{\pi_2}$ be generic tempered irreducible representations of $\G_n(F)$ with the same central characters and the same $\mu_{c'}^{\pm}$-genericity. Suppose that
\[\gamma(s,\wt{\pi_1} \times \rho,\psi)=\gamma(s,\wt{\pi_2} \times \rho,\psi)
\] holds for any irreducible supercuspidal representation $\rho$ of $\GL_t(F)$ with $1\le t \le n$. Then 
\[ \wt{\pi_1} \simeq \wt{\pi_2}.
\]

\end{thm}

\begin{proof}
Let $W_n$ be a $2n$-dimesional symplectic space. Put 
\[(\wt{\pi_1}',\wt{\pi_2}')=\begin{cases} (\wt{\pi_1},\wt{\pi_2}) \quad \quad \quad \quad \quad \quad \ \text{if $\wt{\pi_1},\wt{\pi_2}$ are $\mu_{c'}^{+}$-generic }\\
 (\wt{\pi_1}\otimes \det,\wt{\pi_2}\otimes \det) \quad \text{if $\wt{\pi_1},\wt{\pi_2}$ are $\mu_{c'}^{-}$-generic but not $\mu_{c'}^{+}$-generic}.\end{cases}\]

Then by Proposition~\ref{Theta1}, $\Theta_{\psi,V_n,W_n}(\wt{\pi_1}')$ and $\Theta_{\psi,V_n,W_n}(\wt{\pi_2}')$ are nonzero and $(\mu_{-c'}')^{-1}$-generic. Put $\tau_1:=\theta_{\psi,V_n,W_n}(\wt{\pi_1}')$ and $\tau_2:=\theta_{\psi,V_n,W_n}(\wt{\pi_2}')$. Since $\wt{\pi_1}$ and $\wt{\pi_2}$ have the same central characters, $\tau_1,\tau_2$ also have the same central characters by Proposition~\ref{cen}. Furthermore, by Corollary~\ref{Theta3}, 
\begin{align*}
    \gamma(s,\tau_1 \times \rho,\psi)=\gamma(s,\wt{\pi_1}' \times \rho \chi_{V_n},\psi)\cdot\gamma(s,\rho \chi_{V_n},\psi), \\
\gamma(s,\tau_2 \times \rho,\psi)=\gamma(s,\wt{\pi_2}' \times \rho \chi_{V_n},\psi)\cdot \gamma(s,\rho \chi_{V_n},\psi)
\end{align*} 
for any irreducible supercuspidal representation $\rho$ of $\GL_i(F)$ where $1 \le i \le n$.

Therefore, we obtain 
\[\gamma(s,\tau_1 \times \rho,\psi)=\gamma(s,\tau_2 \times \rho,\psi)\]
and due to property of the dependence on $\psi$ of $\gamma$-factors, i.e., Property \ref{Property_gamma}(iii), we have
\[\gamma(s,\tau_1 \times \rho,\psi^{-1})=\gamma(s,\tau_2 \times \rho,\psi^{-1}).\]
Then by Theorem~\ref{esp}, we have $\tau_1 \simeq \tau_2$ and by Howe duality, $\wt{\pi_1}'\simeq \wt{\pi_2}'$. Henceforth, we have $\wt{\pi_1} \simeq \wt{\pi_2}$.
\end{proof}

The following proposition is analogous to \cite[Proposition 3.2]{JS03} and \cite[Corollary 5.9, \ Proposition 5.11]{Jo}, representing their extension from the supercuspidal case to the tempered case.

\begin{prop}\label{temp}
Let $\G$ be either $\G_n$ or $\J_n$. For an irreducible generic tempered representation $\tau$ of $\G(F)$ and  irreducible unitary supercuspidal representations $\rho$ and $\rho_i$ of $\GL_k(F)$ and $\GL_{k_i}(F)$, respectively, we have the following:
\begin{enumerate}
\item The product $\prod_{i=1}^r \gamma(s+s_i,\rho_i \times \rho,\psi)$ has a real pole $($respectively, a real zero$)$ at $s=s_0$ if and only if $\rho \cong \rho_i^{\vee}$
and $s_0=1-s_i$ (respectively, $s_0=-s_i)$ for some $1 \leq i \leq r$.
\item The product $\prod_{i=1}^r \gamma(s-s_i,\rho_i^{\vee} \times \rho,\psi)$ has a real pole $($respectively, a real zero$)$ at $s=s_0$ if and only if $\rho \cong \rho_i$
and $s_0=1+s_i$ (respectively, $s_0=s_i)$ for some $1 \leq i \leq r$.
\item $\gamma(s,\tau \times \rho,\psi)$ has no zero for $\Re(s) >0$.
\item If $\gamma(s,\tau \times \rho,\psi)$ has a real pole at $s=s_0$, then the pole must be a simple pole at $s_0=1$ and $\rho\simeq \rho^{\vee}$. 
\end{enumerate}
\end{prop}

\begin{proof}
$(\rm i)$ and $(\rm{ii})$ follows a similar line reasoning in \cite[Section 3.2]{JS03}. Note that both $\tau$ and $\rho$ are tempered and generic. From the definition of local $L$-function (\ref{L-ftn}), the zero of $\gamma(s,\tau \times \rho,\psi)$ comes from the poles of $L(s, \tau \times \rho)$. However, Property \ref{Property_gamma}~(vii) implies that $L(s,\tau \times \rho)$ has no pole for $\Re(s)>0$. Therefore $(\rm{iii})$ of the theorem follows. 

When $\G=\J_n$, (iv) follows from \cite[Corollary 5.9]{Jo} and in the proof of \cite[Corollary 5.9]{Jo}, the restriction  that $\tau$ is supercuspidal is given to apply `Casselman--Shahidi Lemma' (\cite[Lemma~{5.8}]{Jo}). However, `Casselman--Shahidi Lemma' also holds for tempered representations by \cite[Proposition 12.3]{Kim}. Except that, other arguments of \cite[Corollary 5.9]{Jo} equally work for tempered representations. This completes the proof of  $(\rm{iv})$ when $\G=\J_n$.

When $\G=\G_n$, suppose that $\tau$ is $\mu_{c'}^{\pm}$-generic and let $\tau' \in \Irr_{\temp}(\G_{n}(F))$ be an element of, which is either $\tau$ or $(\tau \otimes \det)$, such that $\tau'$ be $\mu_{c'}^{\pm}$-generic. Put $\pi=\theta_{\psi,V_{n},W_{n}}(\tau')$. Then by Proposition~{\ref{tem0}} and Proposition~\ref{Theta1}, $\pi$ is nonzero and $(\mu_{-c'}')^{-1}$-generic and tempered. By Corollary~\ref{Theta3}, 
  \beq \nonumber 
  \gamma(s,\tau' \times \rho,\psi)=\begin{cases}\gamma(s,\pi \times \rho\chi_{V_n}^{-1},\psi)\cdot \gamma(s,\rho\chi_{V_n}^{-1},\psi)^{-1}, \quad &\text{if } \G_n \text{ is split}  \\ \gamma(s,\pi \times \rho\chi_{V_n}^{-1},\psi)\cdot \gamma(s,\rho,\psi)^{-1}, \quad &\text{if } \G_n \text{ is non-split}  \end{cases}.\eeq
Proposition~\ref{temp} (i) says that $\gamma(s,\rho,\psi)^{-1}$ and $\gamma(s,\rho\chi_{V_n}^{-1},\psi)^{-1}$ and $\gamma(s,\rho,\psi)^{-1}$ have no pole at $s=s_0>0$. Therefore, Proposition~\ref{temp} (iv) in case $\G=\G_n$ readily follows from the above equality and Proposition~\ref{temp} (iv) in case $\G=\J_n$.
\end{proof}

The following proposition is a direct consequence of Proposition~\ref{temp}, and its proof is essentially the same as  that of \cite[Theorem 5.1]{JS03}.

\begin{prop}
\label{equal}
Let $\wt{\pi_1}$ (resp. $\wt{\pi_2}$) be an irreducible $\mu_{c'}^{\pm}$-generic representations of $\G_n(F)$, which has tempered support $(\rho_1,\cdots,\rho_r;\wt{\pi_{01}})$ (resp. $(\rho_1',\cdots,\rho_{r'}';\wt{\pi_{02}})$) with exponents $(s_1,\cdots,s_r)$ (resp. $(s_1',\cdots,s_{r'}')$).
 Suppose that
\begin{multline}
\label{GeneralDec}
\left[\prod_{i=1}^r \gamma(s+s_i,\rho_i \times \rho,\psi)\cdot \gamma(s-s_i,\rho_i^{\vee} \times \rho,\psi) \right] \cdot\gamma(s,\wt{\pi_{01}} \times \rho,\psi)\\
 =\left[\prod_{i=1}^{r'} \gamma(s+s_i',\rho_i' \times \rho,\psi)\gamma(s-s_i',(\rho_i')^{\vee} \times \rho,\psi) \right] \cdot\gamma(s,\wt{\pi_{02}} \times \rho,\psi)
\end{multline}
for any irreducible unitary supercuspidal representation $\rho$ of ${\GL}_t(F)$ with $1 \leq t \leq n$. Then, $r=r'$ and there exists a permutation $\mathfrak{s}$
of $\{ 1,2,\dotsm,r\}$ such that
\begin{enumerate}[label=$(\mathrm{\roman*})$]
\item $s_i=s'_{\mathfrak{s}(i)}$ and $\rho_i \simeq \rho'_{\mathfrak{s}(i)}$ for all $i=1,2,\dotsm,r$;
\item $\gamma(s,\wt{\pi_{01}} \times \rho,\psi)=\gamma(s,\wt{\pi_{02}} \times \rho,\psi)$ for any irreducible unitary supercuspidal representation $\rho$ of ${\GL}_t(F)$ with $1 \leq t \leq n$. 
\end{enumerate}
\end{prop}

\begin{proof} By putting $\rho=\rho_1$ in the equation (\ref{GeneralDec}), we obtain \begin{multline}
\label{GD1}
\left[\prod_{i=1}^r \gamma(s+s_i,\rho_i \times \rho_1,\psi)\cdot \gamma(s-s_i,\rho_i^{\vee} \times \rho_1,\psi) \right] \cdot\gamma(s,\wt{\pi_{01}} \times \rho_1,\psi)\\
 =\left[\prod_{i=1}^{r'} \gamma(s+s_i',\rho_i' \times \rho_1,\psi)\cdot \gamma(s-s_i',(\rho_i')^{\vee} \times \rho_1,\psi) \right] \cdot\gamma(s,\wt{\pi_{02}} \times \rho_1,\psi).
\end{multline}
By Proposition~\ref{temp}, $\gamma(s-s_1,\rho_1^{\vee}\times \rho_1
,\psi)$ has a pole at $s=s_1+1$ and the left-hand side (LHS) of the equation (\ref{GD1}) has no zero at $s=s_1+1$. Therefore, it has a pole at $s=s_1+1$. The poles on the right-hand side (RHS) of the equation (\ref{GD1}) can arise from one of the following terms:
\begin{enumerate}
    \item $\prod_{1\le i \le r'}\gamma(s+s_i',\rho_i' \times \rho_1,\psi)$,
    \item $\prod_{1\le i \le r'}\gamma(s-s_i',(\rho_i')^{\vee} \times \rho_1,\psi)$; or 
    \item $\gamma(s,\wt{\pi_{02}} \times \rho_1,\psi)$. 
\end{enumerate}

If the pole $s=s_1+1$ on the RHS originates from $\prod_{1\le i \le r'}\gamma(s-s_i',(\rho_i')^{\vee} \times \rho_1,\psi)$, then by Proposition~\ref{temp} (ii), we must have $s_1+1=s_i'+1$ and $\rho_1=\rho_i'$ for some $1\le i \le r'$. Consequently, we can cancel the term
\[
\gamma(s+s_1,\rho_1 \times \rho,\psi)\cdot \gamma(s-s_1,\rho_1^{\vee} \times \rho,\psi)
=
\gamma(s+s_i',\rho_i' \times \rho,\psi)\cdot \gamma(s-s_i',(\rho_i')^{\vee} \times \rho,\psi)
\]
on both sides of the equation (\ref{GeneralDec}). 

For later use in the proof, we call the above argument that cancels gamma factors as Argument A. 

By applying the Argument A iteratively for $s=s_i+1$, there exist some $1\le t \le r$ and a permutation $\mathfrak{p}$ of $\{1,2,\cdots,r'\}$ such that $s=s_j +1$ is a pole of $\prod_{i=1}^{r'-(t-1)}\gamma(s-s_{\mathfrak{p}(i)}',(\rho_{\mathfrak{p}(i)}')^{\vee} \times \rho_t,\psi)$ for $j=1, \cdots, t-1$ and $s=s_t+1$ is not a pole of $\prod_{i=1}^{r'-(t-1)}\gamma(s-s_{\mathfrak{p}(i)}',(\rho_{\mathfrak{p}(i)}')^{\vee} \times \rho_t,\psi)$. Thus, we obtain the following refined equality:
\begin{multline}
\label{GD2}
\left[\prod_{i=t}^r \gamma(s+s_i,\rho_i \times \rho,\psi)\cdot \gamma(s-s_i,\rho_i^{\vee} \times \rho,\psi) \right] \cdot\gamma(s,\wt{\pi_{01}} \times \rho,\psi)\\
 =\left[\prod_{i=1}^{r'-(t-1)} \gamma(s+s_{\mathfrak{p}(i)}',\rho_{\mathfrak{p}(i)}' \times \rho,\psi)\cdot \gamma(s-s_{\mathfrak{p}(i)}',(\rho_{\mathfrak{p}(i)}')^{\vee} \times \rho,\psi) \right] \cdot\gamma(s,\wt{\pi_{02}} \times \rho,\psi)
\end{multline}
Note that $s=s_t+1$ is a pole of either $\prod_{i=1}^{r'-(t-1)} \gamma(s+s_{\mathfrak{p}(i)}',\rho_{\mathfrak{p}(i)}' \times \rho_t,\psi)$  or $\gamma(s,\wt{\pi_{02}}\times \rho_t,\psi)$.

If $s=1+s_t$ is a pole of $\gamma(s,\wt{\pi_{02}}\times \rho_t,\psi)$, then Proposition~{\ref{temp}} (iv) implies that $s_t=0$ and $\rho_t=\rho_t^{\vee}$. 
Thus, $\gamma(s,\rho_t \times \rho_t,\psi)^2$ has a \textit{double pole} at $s=1$, while $\gamma(s,\wt{\pi_{02}}\times \rho_t,\psi)$ has a simple pole at $s=1$. Again, by putting $\rho=\rho_t$ in the equation (\ref{GD2}), we see that there exists a $1\le i \le r'-(t-1)$ such that $\gamma(s+s_{\mathfrak{p}(i)}',\rho_{\mathfrak{p}(i)}' \times \rho_t,\psi)$ has a pole at $s=1$, which implies that $s_{\mathfrak{p}(i)}'=0$ and $\rho_t=(\rho_{\mathfrak{p}(i)}')^{\vee}$. Therefore, if $s=1+s_t$ is a pole of $\gamma(s,\wt{\pi_{02}}\times \rho_t,\psi)$, then we can cancel out $\gamma(s+s_t,\rho_t \times \rho,\psi)\cdot \gamma(s-s_t,\rho_t^{\vee} \times \rho,\psi)=\gamma(s+s_{\mathfrak{p}(i)}',\rho_{\mathfrak{p}(i)}' \times \rho,\psi)\cdot \gamma(s-s_{\mathfrak{p}(i)}',(\rho_{\mathfrak{p}(i)}')^{\vee} \times \rho,\psi)$ on both sides of the equation (\ref{GD2}).

Similarly, for later use in the proof, in case $s$ is a pole of $\gamma(s,\wt{\pi_{02}}\times \rho_t,\psi)$, we call the above argument that cancels gamma factors as Argument B. 

By applying Arguments A and B iteratively, we eventually obtain the following equality:
\begin{multline}
\label{GD4}
\left[\prod_{i=t_1}^r \gamma(s+s_i,\rho_i \times \rho,\psi)\cdot \gamma(s-s_i,\rho_i^{\vee} \times \rho,\psi) \right] \cdot\gamma(s,\wt{\pi_{01}} \times \rho,\psi)\\
 =\left[\prod_{i=1}^{r'-(t_1-1)} \gamma(s+s_{\mathfrak{p}_1(i)}',\rho_{\mathfrak{p}_1(i)}' \times \rho,\psi)\cdot \gamma(s-s_{\mathfrak{p}_1(i)}',(\rho_{\mathfrak{p}_1(i)}')^{\vee} \times \rho,\psi) \right] \cdot\gamma(s,\wt{\pi_{02}} \times \rho,\psi)
\end{multline} for some $t \le t_1 \le r$ and a permutation $\mathfrak{p}_1$ of $\{1,2,\cdots,r'\}$ such that $s=1+s_{t_1}$ is not a pole of $\prod_{i=1}^{r'-(t_1-1)}  \gamma(s-s_{\mathfrak{p}_1(i)}',(\rho_{\mathfrak{p}_1(i)}')^{\vee} \times \rho_{t_1},\psi)$ and $\gamma(s,\wt{\pi_{02}}\times \rho_{t_1},\psi)$.

Since $s=1+s_{t_1}$ is a pole of 
$\gamma(s-s_{t_1},\rho_{t_1}^{\vee} \times \rho_{t_1},\psi)$, it follows that the RHS of the equation (\ref{GD4}) with $\rho=\rho_{t_1}$ has a pole at $s=1+s_{t_1}$. By our choice of $t_1$, it must be a pole of $\prod_{i=1}^{r'-(t_1-1)} \gamma(s+s_{\mathfrak{p}_1(i)}',\rho_{\mathfrak{p}_1(i)}' \times \rho_{t_1},\psi)$. By Proposition~\ref{temp}, we have $1+s_{t_1}=1-s_{\mathfrak{p}_1(i_1)}'$ and $\rho_{t_1}=(\rho_{\mathfrak{p}_1(i_1)}')^{\vee}$ for some $1\le i_1 \le r'-(t_1-1)$. Therefore, $s_{t_1}=s_{\mathfrak{p}_1(i_1)}'=0$. If $s_{\mathfrak{p}_1(k)}'>0$ for some $1\le k \le r'-(t_1-1)$, then the LHS of the equation $(\ref{GD4})$ with $\rho=\rho_{\mathfrak{p}_1(k)}'$ has a pole at $s=1+s_{\mathfrak{p}_1(k)}'$. However, this is impossible because $s_{t_1}=s_{t_1+1}=\cdots=s_r=0$. Therefore, $s_{\mathfrak{p}_1(i)}'=0$ for all $1\le i \le r'-(t_1-1)$ and we have:
\begin{multline}
\label{GD5}
\left[\prod_{i=t_1}^r \gamma(s,\rho_i \times \rho,\psi)\cdot \gamma(s,\rho_i^{\vee} \times \rho,\psi) \right] \cdot\gamma(s,\wt{\pi_{01}} \times \rho,\psi)\\
 =\left[\prod_{i=1}^{r'-(t_1-1)} \gamma(s,\rho_{\mathfrak{p}_1(i)}' \times \rho,\psi)\cdot \gamma(s,(\rho_{\mathfrak{p}_1(i)}')^{\vee} \times \rho,\psi) \right] \cdot\gamma(s,\wt{\pi_{02}} \times \rho,\psi).
\end{multline} 
Since $\rho_{t_1}=(\rho_{\mathfrak{p}_1(i_1)}')^{\vee}$, we can remove $\gamma(s,\rho_{t_1} \times \rho,\psi)\cdot \gamma(s,\rho_{t_1}^{\vee} \times \rho,\psi)=\gamma(s,\rho_{\mathfrak{p}_1(i)}' \times \rho,\psi)\cdot \gamma(s,(\rho_{\mathfrak{p}_1(i)}')^{\vee} \times \rho,\psi)$ on both sides of the equation (\ref{GD5}). In this way, applying Arguments A, B, and the above argument, we can cancel out $\left[\prod_{i=t_1}^r \gamma(s,\rho_i \times \rho,\psi)\cdot \gamma(s,\rho_i^{\vee} \times \rho,\psi) \right]$ in $\left[\prod_{i=1}^{r'-(t_1-1)} \gamma(s,\rho_{\mathfrak{p}_1(i)}' \times \rho,\psi)\cdot \gamma(s,(\rho_{\mathfrak{p}_1(i)}')^{\vee} \times \rho,\psi) \right]$ and we obtain
\begin{equation}
\label{GD6}
 \gamma(s,\wt{\pi_{01}} \times \rho,\psi)=\left[\prod_{i=1}^{r'-r} \gamma(s,\rho_{\mathfrak{p}_2(i)}' \times \rho,\psi)\cdot \gamma(s,(\rho_{\mathfrak{p}_2(i)}')^{\vee} \times \rho,\psi) \right]\cdot\gamma(s,\wt{\pi_{02}} \times \rho,\psi)
\end{equation} for some permutation $\mathfrak{p}_2$ of $\{1,2,\cdots,r'\}$. If $r'>r$, then we put $\rho=\rho_{\mathfrak{p}_2(1)}'$ in the equation (\ref{GD6}). It follows that $\gamma(s,(\rho_{\mathfrak{p}_2(1)}')^{\vee} \times \rho_{\mathfrak{p}_2(1)}',\psi)$ has a pole at $s=1$, and hence $ \gamma(s,\wt{\pi_{01}} \times \rho_{\mathfrak{p}_2(1)}',\psi)$ must also have a pole at $s=1$, which implies $\rho_{\mathfrak{p}_2(1)}'=(\rho_{\mathfrak{p}_2(1)}')^{\vee}$. However, while the RHS of the equation (\ref{GD6}) with $\rho=\rho_{\mathfrak{p}_2(1)}'$ has at least double pole at $s=1$, $ \gamma(s,\wt{\pi_{01}} \times \rho_{\mathfrak{p}_2(1)}',\psi)$ has at most a simple pole at $s=1$. This is a contradiction, and thus we conclude that $r'=r$, completing the proof.
\end{proof}

Now we are ready to prove the following main theorem.
\begin{thm}\label{m}Let $\wt{\pi_1}$ and $\wt{\pi_2}$ be generic irreducible representations of $\G_n(F)$ with the same central character and the same $\mu_{c'}^{\pm}$-genericity. Suppose that
\[\gamma(s,\wt{\pi_1} \times \rho,\psi)=\gamma(s,\wt{\pi_2} \times \rho,\psi)
\] holds for any irreducible supercuspidal representation $\rho$ of $\GL_t(F)$ with $1\le t \le n$. Then 
\[ \wt{\pi_1} \simeq \wt{\pi_2}.
\]
\end{thm}
\begin{proof}Suppose that tempered supports of $\wt{\pi_1}$ and $\wt{\pi_2}$ are  as in Proposition \ref{equal}. Then the central characters of $\wt{\pi_{01}}$ and $\wt{\pi_{02}}$ are same. By Theorem~\ref{tem} and Proposition~\ref{equal}, we have $\wt{\pi_{01}} \simeq \wt{\pi_{02}}$ 
and both $\wt{\pi_1}$ and $\wt{\pi_2}$, therefore, are $\mu_{c'}^{\pm}$-generic irreducible constituents of the induced representation 
\[
\Ind_{\Q(F)}^{\G_n(F)}
(\rho_1|\det|_F^{s_1}\otimes\cdots\otimes \rho_r|\det|_F^{s_r}\otimes \wt{\pi_{01}}).
\] 
Since such induced representation has a unique $\mu_{c'}^{\pm}$-generic constituent, it follows that $\wt{\pi_1} \simeq \wt{\pi_2}$.
\end{proof}

The following theorem is an easy consequence of Theorem~\ref{m}.
\begin{thm}\label{m2}Let $\pi_1$ and $\pi_2$ be $\mu_{c'}$-generic irreducible representations of $\H_n(F)$ with the same central character. Suppose that
\[\gamma(s,\pi_1 \times \rho,\psi)=\gamma(s,\pi_2 \times \rho,\psi)
\] holds for any irreducible supercuspidal representation $\rho$ of $\GL_t(F)$ with $1\le t \le n$. Then 
$[\pi_1]=[\pi_2]$.
\end{thm}

\begin{proof}By Lemma~\ref{gen}, we can take $\wt{\pi_1}$ and $\wt{\pi_2}$ an irreducible $\mu_{c'}^{+}$-generic constituent of $\Ind_{\H_n(F)}^{\G_n(F)}(\pi_1)$ and $\Ind_{\H_n(F)}^{\G_n(F)}(\pi_2)$, respectively. Then by the definition of the $\gamma$-factors for $\G_n \times \GL_t$, we have
\[\gamma(s,\wt{\pi_1} \times \rho,\psi)=\gamma(s,\pi_1 \times \rho,\psi)=\gamma(s,\pi_2 \times \rho,\psi)=\gamma(s,\wt{\pi_2} \times \rho,\psi).
\]
Due to Theorem~\ref{m}, $\wt{\pi_1} \simeq \wt{\pi_2}$ and hence by Proposition~\ref{Clifford}, we have
$[\pi_1]=[\pi_2]$.
\end{proof}

\section{Applications}\label{sec:applications}

By applying our theorem directly, we obtain the weak rigidity theorems for $\O(V)(\A)$ and $\SO(V)(\A)$, respectively. The discussion in this section closely parallels that of \cite[Theorem 5.3]{JS03}. In this section, let $L$ be a number field and $\A$ its ad\'{e}le ring. Let $V$ be a $2n$-dimensional symmetric space over $L$ such that $\O(V)$ is quasi-split. As before, we use the notation $\G_n\coloneqq\O(V)$ and $\H_n\coloneqq \SO(V)$.\par
Let $\pi=\otimes_v \pi_v$ and $\Pi=\otimes_v \Pi_v$ be irreducible automorphic representations of $\SO(V)(\A)$ and $\GL_{2n}(\A)$, respectively. Then we say that $\Pi$ is a weak functorial lift of $\pi$ if $\Pi_v$ is the local Langlands functorial lift of $\pi_v$ for all archimedean places and for almost all places $v$ of $L$, where $\pi_v$ and $\Pi_v$ are unramified.

We are now ready to prove the weak rigidity theorem for $\G_n$.

\begin{thm}[Weak rigidity theorem for $\O(V)$]\label{rigidity_O}
Let $\wt{\chi}_1,\wt{\chi_2}$ be generic characters of $\wt{\U}(L) \bs \wt{\U}(\A)$ such that $\wt{\chi}_1\vert_{\U(\A)}=\wt{\chi}_2\vert_{\U(\A)}$.
Let $\wt{\pi}=\otimes_v \wt{\pi}_v$  (resp. $\wt{\pi}'=\otimes_v \wt{\pi}_v'$) be an irreducible cuspidal globally $\wt{\chi}_1$-generic (resp. $\wt{\chi}_2$-generic) automorphic representation of $\G_n(\A)$. If $\wt{\pi}_v \simeq \wt{\pi}_v'$ or $\wt{\pi}_v \simeq \wt{\pi}_v' \otimes \operatorname{det}$ for almost all $v$, then $\wt{\pi}_v \simeq \wt{\pi}_v'$ or $\wt{\pi}_v \simeq \wt{\pi}_v' \otimes \operatorname{det}$ for all places $v$ of $L$.
\end{thm}
\begin{proof}
Put $\chi=\wt{\chi}_1\vert_{\U(\A)}=\wt{\chi}_2\vert_{\U(\A)}$. Let $\pi$ and $\pi'$ be irreducible cuspidal globally $\chi$-generic automorphic representations of $\H_n(\A)$ that appears in $\wt{\pi} |_{\H_n(\A)}$ and $\wt{\pi}' |_{\H_n(\A)}$, respectively (such $\pi$, $\pi'$ exist by Lemma~\ref{regen}.) The main result in \cite{CPSS11} implies that there exists the weak functorial lift $\Pi$ (resp. $\Pi'$) of $\pi$ (resp. $\pi'$) to $\GL_{2n}(\A)$. The assumption and Definition \ref{Ogamma} imply that for almost all places $v$ we have
\[
\gamma(s, \Pi_v \times \rho, \psi) = \gamma(s, \pi_v \times \rho, \psi) = \gamma(s, \wt{\pi}_v \times \rho, \psi) = \gamma(s, \wt{\pi}_v' \times \rho, \psi) = \gamma(s, \pi_v' \times \rho, \psi) = \gamma(s, \Pi_v' \times \rho, \psi)
\]
for any irreducible supercuspidal representation $\rho$ of $\GL_t(L_v)$ with $1\le t \le n$.
This implies $\Pi_v \simeq \Pi_v'$ for almost all places $v$ due to the local converse theorem for $\GL_{2n}$ (\cite{Cha19}, \cite{JL18}).

Then by the strong multiplicity one theorem for general linear groups, it follows that $\Pi \simeq \Pi'$ and thus $\Pi_v \simeq \Pi_v'$ for all places $v$ of $L$. Then applying the local-to-global argument again exactly as in the proof of \cite[Corollary 4]{CKPSS} together with \cite[Propositions 4.2 and 4.3]{CKPSS}, we have for arbitrary place $v$ of $L$,
\[
\gamma(s, \wt{\pi}_v \times \rho, \psi) = \gamma(s,\pi_v \times \rho ,\psi) = \gamma(s,\Pi_v \times \rho, \psi) = \gamma(s,\Pi_v' \times \rho, \psi) = \gamma(s,\pi_v' \times \rho, \psi) = \gamma(s, \wt{\pi}_v' \times \rho, \psi) 
\]
for any irreducible supercuspidal representation $\rho$ of $\GL_t(L_v)$ with $1\le t \le n$. By the assumption, $\wt{\pi}_v$ and $\wt{\pi}_v'$ are $\mu_{c'}^+$-generic or $\mu_{c'}^-$-generic for some $c'\in L_v^{\times}$. Theorem~\ref{m2} implies that $\wt{\pi}_v \simeq \wt{\pi}_v'$ or $\wt{\pi}_v \simeq \wt{\pi}_v' \otimes \operatorname{det}$. (Note that local functorial lifting at archimedean place is injective for generic representations.)
\end{proof}

The weak rigidity theorem for $\SO(V)$ is a consequence of the weak rigidity theorem for $\O(V)$.
\begin{cor}[Weak rigidity theorem for $\SO(V)$]\label{rSO}
Let $\chi$ be a generic character of $\U(L) \bs \U(\A)$. Let $\pi=\otimes_v \pi_v$ and $\pi'=\otimes_v \pi_v'$ be irreducible cuspidal globally $\chi$-generic automorphic representations of $\H_n(\A)$. If $[\pi_v]= [\pi_v']$ for almost all $v$, then $[\pi_v] = [\pi_v']$ for all places $v$ of $L$.
\end{cor}

\begin{proof}\label{rigidity_SO}
There are generic characters $\wt{\chi}_1,\wt{\chi}_2$ of $\wt{\U}(L) \bs \wt{\U}(\A)$ such that $\wt{\chi}_1\vert_{\U(\A)}=\wt{\chi}_2\vert_{\U(\A)}=\chi$ and globally $\wt{\chi}_1$-generic and $\wt{\chi}_2$-generic cuspidal representations $\wt{\pi},\wt{\pi}'$ of $\G_n(\A)$ such that $\pi$ and $\pi'$ appear in $\wt{\pi}\vert_{\H_n(\A)}$ and $\wt{\pi}'\vert_{\H_n(\A)}$, respectively (such $\wt{\pi}$, $\wt{\pi'}$ exist by Lemma~\ref{exgen}.)  For any finite place $v$ of $L$ such that $[\pi_v]= [\pi_v']$, we have $\Ind_{\H_n(L_v)}^{\G_n(L_v)}\pi_v\simeq \Ind_{\H_n(L_v)}^{\G_n(L_v)}\pi_v'$ by Proposition~\ref{Clifford}. Since $\wt{\pi}_v$ and $\wt{\pi}_v'$ are irreducbile constituent of $\Ind_{\H_n(L_v)}^{\G_n(L_v)}\pi_v$ and  $\Ind_{\H_n(L_v)}^{\G_n(L_v)}\pi_v'$, respectively, we see that \[\wt{\pi}_v\simeq  \wt{\pi}_v' \text{ or } \wt{\pi}_v\simeq \wt{\pi}_v'\otimes \det.\]
Therefore, since $[\pi_v]= [\pi_v']$ for almost all places of $L$, we have $\wt{\pi}_v\simeq \wt{\pi}_v' \text{ or } \wt{\pi}_v\simeq \wt{\pi}_v'\otimes \det$ for almost all places $v$ of $L$, and hence by the weak rigidity theorem for $\O(V)$, $\wt{\pi}_v\simeq \wt{\pi}_v' \text{ or } \wt{\pi}_v\simeq \wt{\pi}_v'\otimes \det$ for all places $v$ of $L$. Since $\pi_v$ and $\pi_v'$ are an irreducible constituent of the restriction of $\wt{\pi_v}$ and  $\wt{\pi_v}'$ to $\H_n(L_v)$, respectively, by Proposition~\ref{Clifford}, we conclude that $[\pi_v]= [\pi_v']$ for all places $v$ of $L$. This completes the proof.
\end{proof}

\begin{rem}
Note that there is direct proof of Theorem \ref{rigidity_O} without using the restriction method. We can directly apply the arguments of the proof of Theorem \ref{rigidity_O} to the case of $\H_n(F)$. Namely, using the existence of the weak functorial lift for $\H_n(\A)$, the multiplicity one theorem for $\GL_{2n}$, and the local-to-global argument, we can prove the weak rigidity for $\SO(V)$.
\end{rem}

\begin{appendix}
\section{Computation of the twisted Jacquet module of the Weil representation}\label{sec:tjm}
In this section, we compute the twisted Jacquet module of the Weil representation which is needed in the proof of Proposition~\ref{Theta1}.
Write $V=V_n$ and define a symplectic form $(\ , \ )$ on $V \otimes W_m$ as follows:
\[(v_1 \otimes w_1, v_2 \otimes w_2)=\la v_1,v_2 \ra_{V}\cdot \la w_1,w_2 \ra_{W_m}.\]
There is a natural embedding of $\G_n \times \J_m$ into $\Sp(V \otimes W_m)$. By pulling back the action of $\wt{\Sp}(V \otimes W_n)$ on the Schr\"odinger model of the Weil representation to $\G_n 
\times \J_m$, we have an action $\omega_{\psi,V,W_m}$ of $\G_n(F) \times \J_m(F)$ on the Schwartz-Bruhat function space $S(V \otimes Y_m^*)(F)$ on $(V \otimes Y_m^*)(F)$. To describe it, let $\P'=\M' \N'$ be a parabolic subgroup of $\J_m$ stabilizing $Y_m$ with Levi subgroup $\M'$. Then 
\[\M'\simeq \GL(Y_m) \text{ and } \N'\simeq \{\alpha \in \Hom(Y_m^*,Y_m)\ \vert \ \alpha^*=-\alpha\},\]
where $\alpha^*$ is the element in $\Hom(Y_m^*,Y_m)$ satisfying 
\[\la \alpha y_1,y_2\ra=\la y_1,\alpha^*y_2\ra, \quad \text{ for all } y_1,y_2 \in Y_m^*.\]
Let $m':\GL(Y_m) \to \M'$ be the isomorphism between $\GL(Y_m)$ and $\M'$. For $a \in \GL(Y_m)$, write $a^*$ the element in $\GL(Y_m^*)$ satisfying \[\la ay,  y^* \ra_{W_m}=\la y,a^*y^*  \ra_{W_m}, \quad  \text{ for all } y \in Y_m, y^* \in Y_m^*.\]
Then the action of $\G_n(F) \times \P'(F)$ on $\omega_{\psi,V_n,W_m}$ is described as follows:
\begin{itemize}
\item $\omega_{\psi,V,W_m}(g,1)\phi(v)=\phi(g^{-1}\cdot v)$ for $g \in \G_n(F)$,
\item $\omega_{\psi,V,W_m}\big(1,m'(a)\big)\phi(v)=\chi_V(\det(a))\cdot |\det(a)|^{n}\cdot \phi\big(a^* \cdot v\big)$ for $a \in \GL(Y_m)(F)$,
\item $\omega_{\psi,V,W_m}(1,n)\phi(v)=\psi\big(\frac{1}{2}( n \cdot v,v)\big)\phi(v)$ for $n \in \N'(F)$,
\end{itemize}
where $v \in (V \otimes Y_m^*)(F)$. \par

Using this action, we can prove the following proposition. 

\begin{prop}\label{cen}
Let $\wt{\pi} \in \Irr(\G_n(F))$. If $\Theta_{\psi,V,W_m}(\wt{\pi})$ is non-zero, then the central characters $w_{\wt{\pi}}$ and $w_{\theta_{\psi,V,W_m} (\wt{\pi})}$ of $\wt{\pi}$ and $\theta_{\psi,V,W_m}(\wt{\pi})$, respectively, are related as $w_{\theta_{\psi,V,W_m}(\wt{\pi})}=w_{\wt{\pi}}\cdot \chi_{V}^{m}$. 
\end{prop}
\begin{proof}
In case $\cha(F)=0$, it is stated in {\cite[Section 5.2]{GI1}}. However, since we were not able to find a reference for $\cha(F)=p\ne 2$ case, we provide the proof which works for $\cha(F)\ne 2$ cases. 

Since the small theta lift $\theta_{\psi,V,W_m}(\wt{\pi})$ of $\wt{\pi}$ is nonzero, we have $\Hom_{\G_n(F) \times \J_m(F)}(\omega_{\psi,V,W_m},\wt{\pi} \otimes \theta_{\psi,V,W_m}(\wt{\pi}) )\ne 0$. Choose a nonzero element $l \in \Hom_{\G_n(F) \times \J_m(F)}(\omega_{\psi,V,W_m},\wt{\pi} \otimes \theta_{\psi,V,W_m}(\wt{\pi}))$. Denote by $I_V$ and $I_{W_m}$ the identity element of $\G_n(F)$ and $\J_m(F)$, respectively. From the above action of $\G_n(F) \times \P'(F)$ on $\omega_{\psi,V,W_m}$, we have
\[(\omega_{\psi,V,W_m}(-I_V,-I_{W_m}) \cdot \phi)(v)=(\chi_V(-1))^{m} \cdot \phi(v), \quad \text{for } \phi \in \omega_{\psi,V,W_m}, \ v \in (V \otimes Y_m^*)(F).\]

On the other hand, for a pure tensor $f_1 \otimes f_2 \in \wt{\pi} \otimes \theta_{\psi,V,W_m}(\wt{\pi})$, $\big(\wt{\pi}\otimes \theta_{\psi,V,W_m}(\wt{\pi})\big)(-I_V,-I_{W_m})(f_1 \otimes f_2)=\omega_{\wt{\pi}}(-I_V) \cdot \omega_{\theta_{\psi,V,W_m} (\wt{\pi})}(-I_{W_m})\cdot (f_1 \otimes f_2).$ Since every element of $\wt{\pi} \otimes \theta_{\psi,V,W_m}(\wt{\pi})$ is a sum of pure tensors, for an arbitrary element $f \in \wt{\pi} \otimes \theta_{\psi,V,W_m}(\wt{\pi})$, we have \[\big(\wt{\pi}\otimes \theta_{\psi,V,W_m}(\wt{\pi})\big)(-I_V,-I_{W_m})\cdot f=\omega_{\wt{\pi}}(-I_V) \cdot \omega_{\theta_{\psi,V,W_m}(\wt{\pi})}(-I_{W_m})\cdot f.\]
Therefore, for any $\phi \in \omega_{\psi,V,W_m}$ such that $l(\phi)\ne 0$, we have
\[(\chi_V(-1))^m\cdot l(\phi)= l((-I_V,-I_{W_m})\cdot \phi)=(-I_V,-I_{W_m})\cdot l(\phi)= \omega_{\wt{\pi}}(-I_V) \cdot \omega_{\theta_{\psi,V,W_m}(\wt{\pi})}(-I_{W_m}) \cdot l(\phi).\]
Since the center of $\G_n$ and $\J_m$ are generated by $-I_V$ and $-I_{W_m}$, respectively, it completes the proof.

\end{proof}

Let $\Z'$ be the maximal unipotent subgroup of $\M'$. Using a fixed basis $\{w_1^*,\cdots,w_m^*\}$ of $Y_m^*$, we may regard $V \otimes Y_m^*$ as $V^m$.
Using the basis $\{w_1,\cdots,w_m\}$ of $Y_m$, 
we can write $\Z'$ as upper triangular matrix group
\[\mf{Z}_m= \left\{\begin{pmatrix} 1  & * & \cdots & * \\ 0 & 1 & \cdots & *\\ \vdots  & 0 & \ddots & \vdots \\ 0 & \cdots & 0 & 1 \end{pmatrix} \in \GL_m\right\}
\]
through an isomorphism $m'$. We denote by $\M_{m \times m}$ the $(m \times m)$ matrix group. Similarly using the basis $\{w_1,\cdots,w_m\}$ of $Y_m$ and $\{w_m^*,\cdots,w_1^*\}$ of $Y_m^*$, we consider $\N'$ as a subgroup $\mf{S}_m$ of $\M_{m \times m}$ and let $n':\mf{S}_m \to \N'$ be the isomorphism between $\mf{S}_m$ and $\N'$. We can then describe the action of $\Z'$ and $\N'$ on $\omega_{\psi,V,W_m}$ in terms of $\mf{Z}_m$ and $\mf{S}_m$ as follows:
\begin{align}
    \label{w1} &\big(\omega_{\psi,V,W_m}\big(1,m'(z)\big)\phi\big)(v_1,\cdots,v_m)= \phi\big( (v_1,\cdots,v_m)\cdot z\big) \text{ for } z \in \mf{Z}_m\\
\label{w2}&\big(\omega_{\psi,V,W_m}(1,n'(s))\phi\big)(v_1,\cdots,v_m)=\psi\big(
\frac{1}{2}\tr\big(Gr(\mathbf{v})\cdot s\cdot \varpi_m \big)\phi(v_1,\cdots,v_m) \text{ for } s \in \mf{S}_m,
\end{align}
where $\mathbf{v}=(v_1,v_2,\ldots,v_m)$, $Gr(\mathbf{v})=\big(\la v_i,v_j \ra_{V_m}\big)$ and $\varpi_m=\begin{pmatrix}0 & &1 \\ & \iddots & \\1 & &0 \end{pmatrix}\in \GL_m$.

Inspired by the proof of \cite[Proposition~2.4, Corollary~2.5]{GRS}, we now prove the following theorem, which is crucial when we generalize the converse proposition of \cite[Corollary~2.5]{GRS} to the case of quasi-split orthogonal groups (Proposition \ref{Theta1}). We also note that \cite[Proposition~9.2]{GS} mentions an analogous statement for metaplectic and odd orthogonal groups when $\cha(F)=0$.

\begin{thm}\label{thm:tjw1} 
Let $(\omega_{\psi,V_n,W_n})_{\U',(\mu_{-c'}')^{-1}}$ be the twisted Jacquet module of $\omega_{\psi,V_n,W_n}$ with respect to $\U'$ and $(\mu_{-c'}')^{-1}$. Then, we have
\[
(\omega_{\psi,V_n,W_n})_{\U',(\mu_{-c'}')^{-1}}\cong \ind_{\wt{\U}}^{\G_n} (\mu_{c'}^{+}).
\]
\end{thm}
\begin{proof}Put 
\[\overline{V_{c'}}=\left\{\mathbf{v}=(v_1,\cdots,v_n) \in V^n \ \vert \ Gr(\mathbf{v})=-2c'\begin{pmatrix} 0 & \cdots & 0 \\ \vdots  & \vdots & \vdots \\ 0 & \cdots & 1\end{pmatrix}\right\},
\]
where $ Gr(\mathbf{v})=\big((v_i,v_j)\big)$. We first claim that 
\[(\omega_{\psi,V_n,W_n})_{\U',(\mu_{-c'}')^{-1}}\simeq S(\overline{V_{c'}}).\]
Note that $\overline{V_{c'}}$ is a closed subset of $V^n$. Therefore, by \cite{BZ}, we have the exact sequence
\[\xymatrix{0 \ar[r] & \mc{S}(V
^n \bs \overline{V_{c'}})  \ar[r]^-{\overline{\text{i}}} & \mc{S}(V^n)  \ar[r]^-{\overline{\text{res}}} & \mc{S}(\overline{V_{c'}}) \ar[r] & 0},\]
where $\overline{\text{i}}$ is induced from the open inclusion map $i:V^n \bs \overline{V_{c'}} \to V^n$ and $\overline{\text{res}}:\mc{S}(V^n) \to \mc{S}(\overline{V_{c'}})$ is the restriction map. Let $J_{\U',(\mu_{-c'}')^{-1}}$ be the twisted Jacquet functor with respect to $\U'$ and $(\mu_{-c'}')^{-1}$. Since the functor $J_{\U',(\mu_{-c'}')^{-1}}$ is exact, we have the exact sequence
\[\xymatrix{0 \ar[r] & J_{\U',(\mu_{-c'}')^{-1}}\big(\mc{S}(V
^n \bs \overline{V_{c'}})\big)  \ar[r] & J_{\U',(\mu_{-c'}')^{-1}}\big(\mc{S}(V^n)\big)  \ar[r] & J_{\U',(\mu_{-c'}')^{-1}}\big(\mc{S}(\overline{V_{c'}})\big) \ar[r] & 0}.\]
By the definition of $\overline{V_{c'}}$, $J_{\U',(\mu_{-c'}')^{-1}}\big(\mc{S}(V^n \bs \overline{V_{c'}})\big)=0$ and $J_{\U',(\mu_{-c'}')^{-1}}\big(\mc{S}(V^n)\big)=\mc{S}(\overline{V_{c'}})$. Therefore, our first claim is proved.\par Note that there is an action of $\G_n \times \mf{Z}_n$ on $\mc{S}(\overline{V_{c'}})$ inherited from $\mc{S}(V^n)$.\par

There is a $(\G_n \times \mf{Z}_n)$-action on $\overline{V_{c'}}$ inherited from the left action of $\G_n \times \J_n$ on $V \times W$ as follows:
\[
(v_1,,\cdots,v_n)\cdot (g,z)=(g^{-1}v_1,,\cdots,g^{-1}v_n)\cdot z, \quad (v_1,\cdots,v_n) \in \overline{V_{c'}}, \ (g,z) \in \G_n \times \mf{Z}_n.
\]
From (\ref{w1}), for $z \in \mf{Z}_n$, $(v_1,,\cdots,v_n)\cdot z=(v_1,z_{12}\cdot v_1+v_2,\cdots,\sum_{i=1}^{k-1} z_{i,k}\cdot v_i +v_k,\cdots,\sum_{i=1}^{n-1} z_{i,n}\cdot v_i +v_n)$.
Therefore, if $v_{k}$ is written as a linear combination of $v_1,v_2,\cdots,v_{k-1}$, then $(v_1,\cdots, v_{k-1},v_k,v_{k+1},\cdots,v_n)$ and
$(v_1,\cdots, v_{k-1},0,v_{k+1},\cdots,v_n)$ are in the same $\mf{Z}_n$-orbit. Therefore, every $\mf{Z}_n$-orbit in $\overline{V_{c'}}$ has the representative of the form 
\[(0,\cdots,0,x_1,0,\cdots,0,x_2,0,\cdots,0,x_j,0,\cdots,0;x)\in \overline{V_{c'}} \subset  V^n,
\]
for some $1 \leq j \leq n-1$ such that $x_i \ne 0$ for all $1\le i \le j$ and for each $2\le k \le j$, $x_k$ (resp. $x$) is not expressed as a linear combination of $\{x_i\}_{1\le i <k}$ (resp. $\{x_i\}_{1\le i \le j}$). 
Furthermore, we cannot take the last element $x$ to be zero since if $(v_1, v_2, \cdots, v_n) \in \overline{V_{c'}}$ is such that $v_n$ is linear combination of $v_1, v_2, \cdots v_{n-1}$ as $v_n = \sum_{i=1}^{n-1} c_iv_i$, then $\langle v_n, v_n\rangle = \sum_{i=1}^{n-1}c_i\langle v_n, v_i\rangle=0$ and it contradicts that $(v_1, v_2, \cdots, v_n) \in \overline{V_{c'}}$. From this, we see that $\{x_1,x_2,\cdots  ,x_{j-1},x_j,x\}$ should be a linearly independent set.
\par
By Witt extension theorem, we can choose more restrictive representatives of the $(\G_n \times \mf{Z}_n)$-orbits of $\overline{V_{c'}}$ as 
\[(0,\cdots,0,e_1,0,\cdots,0,e_2,0,\cdots,0,e_j,0,\cdots,0;e)\in \overline{V_{c'}} \subset  V^n.
\]
(Here, $0\le j \le n-1$ and we set $e_0=0$.)\par
Therefore, there are finite $(\G_n \times \mf{Z}_n)$-orbits in $\overline{V_{c'}}$ and index them by $\{V_{c'}(i)\}_{1\le i \le N}$ so that $\dim(V_{c'}(i)) \le \dim(V_{c'}(j))$ for $i \le j$.

\par Note that for each $j\ge 1$, $V_{c'}(j)$ is a closed subset of $\bigcup_{i \ge j}V_{c'}(i)$ and therefore, we have the exact sequence
\begin{align} \label{ex}\xymatrix{0 \ar[r] & \mc{S}\big(\bigcup_{i \ge j+1}V_{c'}(i)\big)  \ar[r] & \mc{S}\big(\bigcup_{i \ge j}V_{c'}(i)\big)  \ar[r] & \mc{S}(V_{c'}(j)) \ar[r] & 0}.\end{align}
We claim that the Schwartz space on each orbit $V_{c'}(j)$ whose representative is of the form \[(0,\cdots,0,e_1,0,\cdots,0,e_2,0,\cdots,0,e_k,0,\cdots,0;e), \quad \text{for $k<n-1$}\] is zero.\par
Let $V_{c'}(j)$ be an orbit in $\overline{V_{c'}}$ whose representative is $\bar{\mathbf{v}}=(0,\cdots,0,e_1,0,\cdots,0,e_2,0,\cdots,0,e_k,0,\cdots,0;e)$ for some $k<n-1$. Suppose that $V_{c'}(j)$ is non-zero and put $R_{\bar{\mathbf{v}}}$ the stabilizer of $\bar{\mathbf{v}}$ in $\G_n \times \mf{Z}_n$. 
Consider a map 
\[\Phi_{\bar{\textbf{v}}} :\mc{S}(V_{c'}(j)) \to \ind_{R_{\bar{\mathbf{v}}}}^{\G_n \times \mf{Z}_n}\mathbb{I}, \ \varphi \mapsto \Phi_{\bar{\textbf{v}}}(\varphi),
\]
where $\Phi_{\bar{\textbf{v}}}$ is defined by
\[\Phi_{\bar{\textbf{v}}}(\varphi)(g,z)\coloneqq (\omega_{\psi,V,W_n}(g,z)\varphi)(\bar{\textbf{v}}).
\]
It is easy to check that $\Phi_{\bar{\textbf{v}}}$ is a $(\G_n \times \mf{Z}_n)$-isomorphism. Since $k<n-1$, there is a simple root subgroup $J$ of $\mf{Z}_n$ such that $1 \times J$ is a subgroup of $R_{\bar{\mathbf{v}}}$. However, $\mu_{c'}'$ is non-trivial on $J$ and it leads to a contradiction.\par
Therefore, by applying the exact sequence (\ref{ex}) repeatedly, we have
\[\mc{S}(\overline{V_{c'}})\simeq \mc{S}(V_{c'}(j_0)),
\]
where $V_{c'}(j_0)$ has the representative $\bar{\mathbf{v}}'=(e_1,\cdots,e_{n-1};e)$. Let $R_{\bar{\mathbf{v}}'}$ be the stabilizer of $\bar{\mathbf{v}}'$ in $\G_n \times \mf{Z}_n$. Then,

\[R_{\bar{\mathbf{v}}'}=\Biggl\{ \Biggr( \biggr( \begin{pmatrix}z &a&*&* \\ 
&1&0&*\\
&&1&a'\\&&&(z^*)^{-1}\end{pmatrix}, \e \biggr), \begin{pmatrix}z &a&& \\ 
&1&&\\
&&1&a'\\&&&(z^*)^{-1}\end{pmatrix} \Biggl) \in \G_n \times \J_n \ \vert \ \Biggl\},
\]
where we described the elements of $\G_n$ amd $\J_n$ using the basis $\{e_1,\cdots,e_{n-1},e,e',e_{n-1}^*,\cdots,e_{1}^*\}$ and $\{w_1,\cdots,w_n,w_n^*,\cdots,w_1^*\}$, respectively.

\par
Then for $\wt{u}=\biggr( \begin{pmatrix}z &a&*&* \\ 
&1&0&*\\
&&1&a'\\&&&z'\end{pmatrix}, \e \biggr) \in \wt{\U}$ and $\phi \in \mc{S}(V_{c'}(j_0))\simeq \ind_{R_{\bar{\textbf{v}}'}}^{\G_n \times \mf{Z}_n}\mathbb{I}$, we have
\[\phi(\wt{u} \cdot g,\textbf{1})=\psi(z_{1,2}+\cdots+z_{n-2,n-1}+a_{n-1})\phi(g,\textbf{1})=\mu_{c'}^{+}\cdot(\wt{u})\cdot \phi(g,\textbf{1}) \quad \text{for all } g \in \G_n.\]
(Here, we view $\phi$ as a function on $\G_n \times \mf{Z}_n$ via $\Phi_{\bar{\textbf{v}}'}$.) This proves Theorem~\ref{thm:tjw1}.
\end{proof}

\section{Non-vanishing and cuspidality of global theta lifts}\label{sec:gtl}
Let $L$ be a global field of characteristic zero or $p$ not equal to 2 and $\A$ be its adele ring. Let $(\VV_n,\la \ ,\ \ra_{\VV_n})$ be a $2n$-dimensional orthogonal space over $L$.
Let $\HH$ be the hyperbolic plane over $L$, i.e. the split symplectic space of dimension 2,
and we consider the Witt tower $\WW_k=\HH^{\oplus k}$. Write $\WW_k=Y_k \oplus Y_k^{*}$, where $Y_k$ and $Y_k^*$ are maximal isotropic subspaces of $\WW_k$ which are dual with respect to the symplectic form $\la \ ,\ \ra_{\WW_k}$ of $\WW_k$ satisfying $0=Y_0\ss Y_1 \ss \cdots \ss Y_k$ and $0=Y_0^*\ss Y_1^* \ss \cdots \ss Y_k^*$. As in the local case, we use the same symbol $\G_n$ (resp. $\J_k$) to denote the isometric group of $\VV_n$ (resp. $\WW_k$.) 

In this section, we state theorem and lemmas concerning the non-vanishing and cuspidality of global theta lifts from $\G_n(\A)$ to $\J_m(\A)$.

We can define the global Weil representation $\omega_{\psi,\VV_n,\WW_m}:=\bigotimes_v \omega_{\psi_v,\VV_{n,v},\WW_{m,v}}$ of $\G_n(\A) \times \J_m(\A)$. Then it is realized in the Schwartz-Bruhat space $S(\VV_n \otimes Y_m^*)(\A)=\bigotimes _v S(\VV_{n,v} \otimes Y_{m,v}^*)(L_v)$.
Define a symplectic form $(\ , \ )$ on $\VV_n \otimes \WW_m$ as follows;
\[(v_1 \otimes w_1, v_2 \otimes w_2)=\la v_1,v_2 \ra_{\VV_n}\cdot \la w_1,w_2 \ra_{\WW_m}.\] Let $\P'=\N'\M'$ be a parabolic subgroup of $\J_m$ stabilizing $Y_m$ with Levi subgroup $\M'$. Then 
\[\M'\simeq \GL(Y_m) \text{ and } \N'\simeq \{\alpha \in \Hom(Y_m^*,Y_m)\ \vert \ \alpha \in \Hom(Y_m^*,Y_m)\ \vert \ \alpha^*=-\alpha \},\]
where $\alpha^*$ is the element in $\Hom(Y_m^*,Y_m)$ satisfying 
\[\la ay_1,y_2\ra=\la y_1,a^*y_2\ra, \quad \text{ for all } y_1,y_2 \in Y_m^*.\]
Then from the action of the (local) Weil representation, we have the action of $\G_n(\A)\times \N'(\A)$ on $\omega_{\psi,\VV_n,\WW_m}$ as follows:
\begin{itemize}
\item $\omega_{\psi,\VV_n,\WW_m}(g,1)\phi(v)=\phi(g^{-1}\cdot v)$ for $g \in \G_n(\A)$
\item $\omega_{\psi,\VV_n,\WW_m}(1,n)\phi(v)=\psi(\frac{1}{2}(n\cdot v, v))\phi(v)$ for $n \in \N'(\A)$,
\end{itemize}
where $v \in (\VV_n \otimes Y_m^*)(\A)$. 

There is an equivariant map $\theta_{\psi,\VV_n,\WW_m}:S((\VV_n \otimes Y_m^*)(\A)) \to \mc{A}(\G_n \times \J_m)$ given by the theta series
\[\theta_{\psi,\VV_n,\WW_m}(\phi)(g,h):=\sum_{y \in (\VV_n \otimes Y_m^*)(L)}\omega_{\psi,\VV_n,\WW_m}(\phi)(g,h)(y).
\]
For an automorphic form $f$ of $\G_n(\A)$, put \[\theta_{\psi,\VV_n,\WW_m}(\phi,f)(h)=\int_{\G_n(L) \bs \G_n(\A)} \theta_{\psi,\VV_n,\WW_m}(\phi;g,h)\overline{f(g)}dg\]
and for an automorphic representation $\pi$ of $\G_n(\A)$, write $\Theta_{\psi,\VV_n,\WW_m}(\pi)=\{\theta_{\psi,\VV_n,\WW_m}(\phi,f) \ \vert \ \phi \in \omega_{\psi,\VV_n,\WW_m}, \ f\in \pi \}$. Then $\Theta_{\psi,\VV_n,\WW_m}(\pi)$ is an automorphic representation of $\J_m(\A)$.

\begin{thm}\label{a1} Let $\pi$ be an irreducible cuspidal representation of $\G_n(\A)$. Then there is a positive integer $t$ such that $\Theta_{\psi,\VV_n,\WW_i}(\pi)=0$ for all $0 \le i <t$ and $\Theta_{\psi,\VV_n,\WW_t}(\pi)\ne 0$. Furthermore, $\Theta_{\psi,\VV_n,\WW_t}(\pi)$ is cuspidal. 
\end{thm}

We believe that the above theorem is almost certainly well known to experts. In particular, it
should be noted that \cite{Ra84} deals with a similar statement. However, we are not able to find a reference for positive characteristic cases and therefore provide a proof for completeness. We hope this proof will be a useful reference for readers interested in positive characteristic cases.

To prove this, we need two lemmas.
\begin{lem}\label{l1}Let $f \in \mc{A}_{cusp}(\G_n)$. Let $\P_k'=\N_k'\M_k'$ be a parabolic subgroup of $\J_m$ stabilizing $Y_k$ so that $\M_k' \simeq \GL(Y_k) \times \J_{m-k}$. Then for all $h \in \J_{m-k}(\A)$,
\[\int_{\N_k'(K) \bs \N_k'(\A)}\theta_{\psi,\VV_n,\WW_{m}}(\phi,f)(nh)dn=\theta_{\psi,\VV_n,\WW_{m-k}}(\wt{\phi},f)(h),
\]
where $\wt{\phi}$ is the restriction of $\phi$ via the natural inclusion map $\VV_n \otimes Y_{m-k}^* \hookrightarrow \VV_n \otimes Y_{m}^*$.

\end{lem}
\begin{proof} For the case of $\cha(L)=0$, it is proved in \cite[Theorem I.1.1]{Ra84}. Since the proof for the case of $\cha(L)=p$ is the same with that of $\cha(L)=0$, we omit it.
\end{proof}

\begin{lem}\label{l2}Let $f \in \mc{A}_{cusp}(\G_n)$. If \[\int_{\G_n(L) \bs \G_n(\A)} \theta_{\psi,\VV_{n},\WW_{2n}}(\phi;g,h)\overline{f(g)}dg=0\]
for all $\phi \in \omega$, then $f=0$.
\end{lem}

\begin{proof}The proof for the case $\mathrm{char}(L)=0$ is implicitly included in \cite[Theorem I.2.1]{Ra84}. To demonstrate that the argument also holds for $\mathrm{char}(L)=p$, we provide a detailed proof. We do not claim originality for this proof.


Put $\mf{S}:=\{A \in M_{2n \times 2n}\ \vert \ A^t=-A \}$. Using bases $\beta=\{f_1,\cdots ,f_{2n}\}$ of $Y_{2n}$ and $\beta^*=\{f_1^*,\cdots ,f_{2n}^*\}$ of $Y_{2n}^*$ satisfying $\la f_i , f_j^* \ra_{\WW_{2n}}=\delta_{ij}$, it is easy to see that
\[  \{\alpha \in \Hom(Y_{2n}^*,Y_{2n})\ \vert \ \alpha^*=-\alpha \}\simeq \mf{S}.\] Let $n':\mf{S} \rightarrow \N'$ be the isomorphism between $\mf{S}$ and $\N'$ and write $J=(\la e_i,e_j\ra_{\VV_n})$ for some basis $\alpha=\{e_1,\cdots,e_{2n}\}$ of $\VV_n$. 

For a function $\varphi$ of $\J_{2n}(\A)$, define its Fourier coefficient $\varphi^{\N',J}$ with respect to $(\N',J)$ by
\[
\int_{\mf{S}(L)\bs \mf{S}(\A)} \varphi(n'(s)h)\psi(-\frac{1}{2}\tr(Js))ds,
\]
where $ds$ is the Haar measure of $\mf{S}(\A)$ such that $\text{Vol}(\mf{S}(L)\bs \mf{S}(\A))=1$.
Note that 
\begin{align*}\theta_{\psi,\VV_n,\WW_{2n}}(\phi)^{\N',J}(g,h)&=\int_{\mf{S}(L)\bs \mf{S}(\A)}\theta_{\psi,\VV_n,\WW_{2n}}(\phi)(g,n'(s)h)\psi(-\frac{1}{2}\tr(Js))ds\\
&=\int_{\mf{S}(L)\bs \mf{S}(\A)}\sum_{y^* \in (\VV_n \otimes Y_{2n}^{*})(L)}\omega_{\psi,\VV_n,\WW_m}(g,n'(s)h)\phi(y^*)\psi(-\frac{1}{2}\tr(Js))ds\\
&=\int_{\mf{S}(L)\bs \mf{S}(\A)}\sum_{y^* \in (\VV_n \otimes Y_{2n}^{*})(L)}\omega_{\psi,\VV_n,\WW_m}(g,h)\phi(y^*)\psi(\frac{1}{2}(n'(s)y^*,y^*))\psi(-\frac{1}{2}\tr(Js))ds.\\
\end{align*}
Using the basis $\beta^*$ of $Y_{2n}^*$ and the basis $\alpha$ of $\VV_n$, we view $y^* \in (\VV_n \otimes Y_{2n}^*)$ as a $(2n \times 2n)$ matrix. Note that $(n'(s)y^*,y^*)=\tr((y^*)^t Jy^*s)$. Therefore, \[
\int_{\mf{S}(L)\bs \mf{S}(\A)}\psi(\frac{1}{2}(n'(s)y^*,y^*))\psi(-\frac{1}{2}\tr(Js))ds=\begin{cases}1 \quad \text{ if } (y^*)^tJy^*=J\\ 0 \quad \text{ otherwise}\end{cases}.
\]
Since $\G_n=\{y^* \in \VV_n \otimes Y_{2n}^* \ \vert \ (y^*)^tJy^*=J\}$, we have 
\[\theta_{\psi,\VV_n,\WW_{2n}}(\phi)^{\N',J}(g,h)=\sum_{y^* \in \G_n(L)}\omega(g,h)\phi(y^*).
\]

Denote by $\mc{S}(\G_n(\AA))$ (resp. $\mc{S}(\G_n(L) \bs \G_n(\A))$) the Schwartz-Bruhat space on $\G_n(\AA)$ (resp. $\G_n(L) \bs \G_n(\A)$.) Because the restriction map $S(\VV_n \otimes Y_{2n}^*) \mapsto \mc{S}(\G_n(\AA))$ is surjective and every function in $\mc{S}(\G_n(L) \bs \G_n(\A))$ is obtained by averaging a function in $\mc{S}(\G_n(\AA))$, we see that $\{\theta_{\psi,\VV_n,\WW_{2n}}(\phi)^{\N',J}(\cdot,1)\}_{\phi \in S(\VV_n \otimes Y_{2n}^*)}$ forms $\mc{S}(\G_n(L) \bs \G_n(\A))$.

By the assumption, we have $\la \theta_{\psi,\VV_n,\WW_{2n}}(\phi)^{\N',J}(\cdot,1) ,  f\ra_{L^2(\G_n(L)\bs \G_n(\A))}=0.$ Since $\mc{S}(\G_n(L) \bs \G_n(\A))$ is a dense subspace of $L
^2(\G_n(L) \bs \G_n(\A))$, it follows that $f=0$.
\end{proof}
Now we can prove Theorem \ref{a1}.

\begin{proof}Lemma \ref{l2} tells us that $\Theta_{\psi,\VV_n,\WW_{2n}}(\pi)\ne 0$ and this proves the first statement. The second statement follows from Lemma \ref{l1}.
\end{proof}
\end{appendix}

\subsection*{Acknowledgements}
This paper builds upon the work of Jo and Zhang, for which we are grateful. We also extend our appreciation to Atobe, Gan, Ichino, and Kaplan for providing the tools utilized in this paper. Special thanks to Kaplan for his assistance in the proof of Lemma~\ref{conj} and kind suggestion to add the case $\cha(F)=p$ in our earlier draft, which dealt only $\cha(F)=0$ case. We also thank to Atobe for his very helpful comments.

The first author expresses deep gratitude to Dongho Byeon, Youn-Seo Choi, Youngju Choie, Wee Teck Gan, Haseo Ki, Sug Woo Shin and Shunsuke Yamana for their constant support and encouragement. He is also grateful to Yonsei university for providing a wonderful place to conduct research and its generous support. The first author has been supported by the National Research Foundation of Korea (NRF) grant funded by the Korea government (MSIT) (No. 2020R1F1A1A01048645). The second author is grateful to Purdue University for providing excellent working conditions during a one-year research visit (July 2022 - July 2023). The second author has been supported by the National Research Foundation of Korea (NRF) grant funded by the Korea government (MSIP) (No. RS-2022-0016551 and  No. RS-2024-00415601 (G-BRL)) and by Chonnam National University (Grant number: 2022-0123). The first and third authors have been supported by NRF grant (No. RS-2023-00237811).

Lastly, we are very grateful to the anonymous referee for his/her thorough review and invaluable comments, which undoubtedly improved the clarity of the paper's presentation.

\providecommand{\bysame}{\leavevmode\hbox to3em{\hrulefill}\thinspace}


\begin{thebibliography}{99}

\bibitem[AB95]{AB}
{J. Adams and D. M. Barbasch},
{\em Reductive Dual Pair Correspondence for Complex Groups},
{\it J. Funct. Anal.} {\bf 132} (1995) 1--42.

\bibitem[Ar13]{Ar}
{J. Arthur},
{\em The endoscopic classification of representations. Orthogonal and symplectic groups}, 
{\it Amer. Math. Soc. Colloq. Publ.} {\bf 61} (2013).

\bibitem[At17]{At17} 
H. Atobe, 
\emph{On the uniqueness of generic representations in an $L$-Packet}, 
{\it Int. Math. Res. Not. IMRN} {\bf{23}} (2017) 7051--7068.

\bibitem[At18]{At18} 
H. Atobe, 
\emph{The local theta correspondence and the local Gan–Gross–Prasad conjecture for the symplectic-metaplectic case}, 
{\it Math. Ann.} {\bf 371} (2018) 225--295.


\bibitem[AG17]{AG17} 
H. Atobe and W. T. Gan, 
\emph{On the local Langlands correspondence and Arthur conjecture for even orthogonal groups}, 
{\it Represent. Theory} {\bf 21} (2017) 354--415.

\bibitem[BJ01]{BJ01}
D. Ban and C. Jantzen, 
{\em The Langlands classification for non-connected $p$-adic groups},
{\it Israel J. Math.} {\bf 126} (2001) 239--261.

\bibitem[BJ03]{BJ}
D. Ban and C. Jantzen, 
{\em Degenerate principal series for even-orthogonal groups},
{\it Represent. Theory} {\bf 7} (2003) 440--480.

\bibitem[BZ76]{BZ}
J. H. Bernstein and A. V. Zelevinskii, 
{\em Representations of the group GL(n,F) where F is a non-Archimedean local field},
{\it Russian Mathematical Surveys}
{\bf 31}, Issue 3 (1976) 1–68.
\bibitem[CFK22]{CFK} 
Y. Cai, S. Friedberg, and E. Kaplan, 
\emph{The generalized doubling method: local theory}, 
{\it Geom. Funct. Anal.} {\bf 32} (2022) 1233--1333.

\bibitem[Ca22]{Ca22}
H. Castillo,
{\em Langlands functoriality conjecture for $\SO_{2n}^*$ in positive characteristic}
\href{https://arxiv.org/abs/2201.03119v1}{arXiv:2201.03119v1}

\bibitem[CGL24]{CGL24}
H. Castillo, G. Henniart, L. Lomelí,
{\em On generic representations of quasi-split reductive groups over local fields of positive characteristic}
\href{https://arxiv.org/abs/2412.00229v1}{arXiv:2412.00229v1}

\bibitem[CS98]{CS} 
W. Casselman and F. Shahidi, 
\emph{On irreducibility of standard modules for generic representations}, 
{\it Ann. Sci. \'{E}cole Norm. Sup.} {\bf 31} (1998) 561--589.

 \bibitem[C19]{Cha19} 
 J. Chai, 
 \emph{Bessel functions and local converse conjecture of Jacquet}, 
 {\it J. Eur. Math. Soc.} {\bf 21} (2019) no. 6 1703--1728.

 \bibitem[CKPSS01]{CKPSS} 
 J. W. Cogdell, H. Kim, I.I. Piathtski-Shapiro, and F. Shahidi, 
 \emph{On lifting from classical groups to $\GL(n)$}, 
 {\it Publ. Math. Inst. Hautes \'{E}tudes Sci.} {\bf 93} (2001) 5--30. 

 \bibitem[CPSS11]{CPSS11} 
  J. W. Cogdell, I. I. Piatetski-Shapiro, and F. Shahidi, \emph{Functoriality for the quasisplit classical
groups}, {\it In On certain L-functions, volume 13 of Clay Math. Proc}, 117–140. Amer. Math.
Soc., Providence, RI, 2011.
[CM93] D. H. Collingwood and W. M

\bibitem[CST17]{CST17} 
J. W. Cogdell, F. Shahidi, and T.-L. Tsai, 
\emph{Local Langlands correspondence for $\GL_n$ and the exterior and symmetric square $\varepsilon$-factors}, 
{\it Duke Math. J.} {\bf 166} (2017) no. 11 2053--2132.

\bibitem[Fu95]{Fu} M. Furusawa, \emph{On the theta lift from $SO_{2n+1}$ to $\widetilde{Sp}_n$}, {\it J. reine angew. Math.} {\bf 466} (1995), 87--110.

\bibitem[GGP12]{GGP} 
W. T. Gan, B. Gross, and D. Prasad, 
\emph{Symplectic local root numbers, central critical $L$ values, and restriction problems in the representation theory of classical groups}, 
{\it Ast\'{e}risque} {\bf 346} (2012) 1--109.

\bibitem[GI14]{GI1}
{W. T. Gan and A. Ichino}, 
{\em Formal degrees and local theta correspondence}, 
{\it Invent. Math.} {\bf 195} (2014) no. 3, 509--672. 

\bibitem[GL18]{GL}
{W. T. Gan and L. Lomelí}, 
{\em Globalization of supercuspidal representations over function fields and applications},
{\it J. Eur. Math. Soc.} {\bf 20} (2018), no. 11, 2813-–2858.

\bibitem[GS12]{GS} W. T. Gan and G. Savin, \emph{Representations of metaplectic groups $\mathrm{I}$: epsilon dichotomy and local Langlands correspondence}, Compositio. M, \textbf{148} (2012), 1655--1694.

\bibitem[GT16a]{GT1}
{W. T. Gan and S. Takeda},
{\em On the Howe duality conjecture in classical theta correspondence},
{\it Advances in the theory of automorphic forms and their $L$-functions}, 
{\it Contemp. Math.} {\bf 664} Amer. Math. Soc. Providence RI (2016) 105--117.

\bibitem[GT16b]{GT2}
{W. T. Gan and S. Takeda},
{\em A proof of the Howe duality conjecture},
{\it J. Amer. Math. Soc.} {\bf 29} (2016) no.~2 473--493.

\bibitem[GRS97]{GRS} D. Ginzberg, S. Rallis, and D. Soudry, \emph{Periods, poles of $L$-functions and symplectic-orthogonal theta lifts},  J. Reine Angew. Math., \textbf{487}, (1997), 85--114.

\bibitem[H25]{H} 
J. Haan, 
\emph{The local converse theorem for $\Mp_{2n}$ : the generic case}, 
manuscripta mathematica, \textbf{176:79}, (2025), 1--30.

\bibitem[Ha25]{Ha} A. Hazeltine, 
\emph{A local converse theorem for quasi-split even special orthogonal groups}, 
\href{https://arxiv.org/pdf/2501.16339}{arXiv:2501.16339}

\bibitem[HL25]{HL} A. Hazeltine and B. Liu, 
\emph{On the local converse theorem for split $SO_{2n}$}, 
Represent. Theory \textbf{25} (2024), 209-255.

\bibitem[HO13]{HO13}
V. Heiermann, E. Opdam,
{\em ON THE TEMPERED L-FUNCTIONS CONJECTURE}, 
{\it American Journal of Mathematics}, vol. 135, no. 3, (2013), pp. 777–799. 

\bibitem[He83]{He} 
G. Henniart,
{\em La conjecture de Langlands locale pour $GL(3)$}, 
{\it Mém. Soc. Math. Fr.} Series 2 no. {\bf 11-12} (1983) 1--188.

\bibitem[He93]{He1} 
G. Henniart, 
\emph{Caract\'{e}risation de la correspondance de Langlands locale par les facteurse de paires}, 
{\it Invent. M.} {\bf 113} (1993) 339--350.


\bibitem[JL18]{JL18} 
H. Jacquet and B. Liu,
\emph{On the local converse theorem for $p$-adic $\GL_n$}, 
{\it Amer. J. Math.} {\bf 140} (2018) no. 5 1399--1422.

\bibitem[Ji06]{Jng}
D. Jiang, 
{\em On local $\gamma$-factors}, In Arithmetic geometry and number theory edited by L. Weng and I. Nakamura, 
{\it World Sci. Publ. Ser. Number Theory Appl.} {\bf 1}, Hackensack NJ (2006) 1--28. 

\bibitem[JSS83]{JSS83} H. Jacquet, I.I. Piatetski-Shapiro, and J.A. Shalika,
\emph{Rankin--Selberg convolutions},
{\it Amer. J. Math.} {\bf 105} (1983) 367--464.

\bibitem[JacS90]{JacS90}
{H. Jacquet and J.A. Shalika}, 
{\em Rankin–Selberg convolutions: Archimedean theory}, 
In Festschrift in honor of I. I. Piatetski-Shapiro on the occasion of his sixtieth birthday, Part $\mathrm{I}$ (Ramat Aviv, 1989) 
{\it Israel Math. Conf. Proc.} {\bf 2}, Weizmann, Jerusalem, (1990) 125--207.
   
\bibitem[JS03]{JS03}
D. Jiang and D. Soudry, 
\emph{The local converse theorem for ${\rm SO}(2n+1)$ and applications}, 
{\it Ann. of Math. (2)} {\bf 157} (2003) no.3 743--806.

\bibitem[Jo22]{Jo} 
Y. Jo, 
\emph{The local converse theorem for odd special orthogonal and symplectic groups in positive characteristic}, 
\href{https://arxiv.org/pdf/2205.09004v1}{arXiv:2205.09004}

\bibitem[Ka15]{Kap15} 
E. Kaplan, 
\emph{Complementary results on the Rankin-Selberg gamma factors of classical groups}, 
{\it J. Number Theory} {\bf 146} (2015) 390--447.

\bibitem[Ki04]{Kim} 
H. Kim, 
\emph{Automorphic $L$-functions}, 
Lectures on automorphic $L$-functions, 
{\it Fields Inst. Monogr.} {\bf 20}, Amer. Math. Soc. Providence RI (2004) 97--201.



\bibitem[Ku86]{Ku86}
{S. S. Kudla}, 
{\em On the local theta correspondence}, 
{\it Invent. Math.} {\bf 83} (1986) 229--255.

\bibitem[Ku]{Ku}
{S. S. Kudla}, 
{\em Notes on the local theta correspondence}, \href{http://www.math.utoronto.ca/~skudla/castle.pdf}{http://www.math.utoronto.ca/~skudla/castle.pdf}.

\bibitem[LR05]{LR05}
E.~M. Lapid and S.~Rallis, 
\emph{On the local factors of representations of classical groups},
\newblock In J.~W. Cogdell, D.~Jiang, S.~S. Kudla, D.~Soudry, and R.~Stanton, editors, {\em Automorphic representations, ${L}$-functions and applications: progress and prospects}, 
{\it Ohio State Univ. Math. Res. Inst. Publ.} {\bf 11}, de Gruyter, Berlin, (2005) 309--359.

\bibitem[L15]{L15} 
{L. Lomeli}, 
\emph{The LS method for the classical groups in positive characteristic and the Riemann Hypothesis}, 
American Journal of Mathematics 137(2), 473-496.

\bibitem[L17]{L17}
{L. Lomeli}, 
\emph{The Langlands-Shahidi method over function fields: Ramanujan Con-
jecture and Riemann Hypothesis for the unitary groups},
\href{https://arxiv.org/pdf/1507.03625}{arXiv:1507.03625}

\bibitem[MS20]{MS20}
{G. Mui\'c and G. Savin},
{\em Symplectic-orthogonal theta lifts of generic descrete series}, {\it Duke Math. J}, {\bf 101}, (2000), 317--333.

\bibitem[M18]{M18}
K. Morimoto, 
\emph{On the irreducibility of global descents for even unitary groups and its applications}, 
{\it Trans. Amer. Math. Soc.} {\bf 370} (2018) 6245--6295.

\bibitem[PS08]{PS}
D. Prasad and R. Schulze-Pillot, 
\emph{Generalised form of a conjecture of Jacquet and a local consequence},
{\it J. reine angew. Math.},
{\bf no.616},
(2008) 219--236.

\bibitem[Ra84]{Ra84}
{S. Rallis},
{\em On the Howe duality conjecture},
{\it Compositio. Math}, {\bf 51(3)}, (1984), 333--399.

\bibitem[RS05]{RS}
{S. Rallis and D. Soudry},
{\em Stability of the local gamma factor arising from the doubling method},
{\it Math. Ann},
{\bf 333},
(2005), 291–313.

\bibitem[S90]{Sha90} 
F. Shahidi, 
\emph{A proof of Langlands' conjecture on Plancherel measures; complementary series for $p$-adic groups}, 
{\it Ann. of Math. (2)} {\bf 132} (1990), 273--330.

\bibitem[YZ23]{YZ23} 
P. Yan and Q. Zhang, 
\emph{Product of Rankin-Selberg convolutions and a new proof of Jacquet’s local converse conjecture}, \href{https://arxiv.org/abs/2309.10445}{arXiv:2309.10445}

\bibitem[YZ24]{YZ24} 
P. Yan and Q. Zhang, 
\emph{On a refined local converse theorem for $\SO(4)$}, 
{\it Proc. Amer. Math. Soc.} {\bf 152} (2024), 4959-4976.


\bibitem[Z18]{Q18}
Q. Zhang, 
\emph{A local converse theorem for ${\rm Sp}_{2r}$},
{\it Math. Ann.} {\bf 372} (2018) no. 1-2 451--488.

\bibitem[Z19]{Q19}
Q. Zhang, 
\emph{A local converse theorem for ${\rm U}_{2r+1}$},
{\it Trans. Amer. Math. Soc.} {\bf 371} (2019) no. 8 5631--5654.

\bibitem[W90]{Wa90}
{J.-L. Waldspurger}, 
{\em Démonstration d’une conjecture de dualité de Howe dans le cas $p$-adique, $p\not= 2$}, 
In Festschrift in honor of I. I. Piatetski-Shapiro on the occasion of his sixtieth birthday, Part $\mathrm{I}$ (Ramat Aviv, 1989) 
{\it Israel Math. Conf. Proc.} {\bf 2}, Weizmann, Jerusalem, (1990) 267--324. (French).
\end{thebibliography}
\end{document}